\documentclass[12pt]{article}
\usepackage{amsthm,amssymb}
\usepackage{amsmath}

\newtheorem{theorem}{Theorem}[section]

\newtheorem{lemma}[theorem]{Lemma}

\newtheorem*{thmA}{Theorem A}
\newtheorem*{thmB}{Theorem B}
\newtheorem*{thmC}{Theorem C}
\newtheorem*{thmD}{Theorem D}
\newtheorem*{thmE}{Theorem E}

\def\irr#1{{\rm Irr} (#1)}

\def\nor{\unlhd\,}
\def\syl#1#2{{\rm Syl}_{#1}(#2)}
\def\cent#1#2{{\bf C}_{#1}(#2)}

\newcommand{\zent}[1]{{\bf Z} (#1)}
\newcommand{\oh}[2]{{\rm O}_{#1} (#2)}
\newcommand{\Oh}[2]{{\rm O}^{#1} (#2)}
\newcommand{\cd}[1]{{\rm cd} (#1)}
\newcommand{\V}[1]{{\rm V} (#1)}
\newcommand{\sbs}{\leq}
\newcommand{\hyp}[2]{(#1,#2)-{\rm hypothesis}}
\newcommand{\subdash}[2]{#1_{{#2}'}}

\newcommand{\sps}{\geq}
\newcommand{\lin}[1]{{\rm Lin} (#1)}
\newcommand{\cdover}[3]{{\rm cd}_{#3} (#1 \mid #2)}
\newcommand{\scr}[1]{{\cal {#1}}}
\newcommand{\cdsub}[2]{{\rm cd}_{#1} (#2)}
\newcommand{\cdsup}[2]{{\rm cd}^{#1} (#2)}
\newcommand{\snor}{{\unrhd\unrhd}}
\newcommand{\st}{\mid}
\newcommand{\ibr}[1]{{\rm IBr} (#1)}
\newcommand{\aut}[1]{{\rm Aut} (#1)}
\newcommand{\cha}[1]{{\rm Char} (#1)}

\begin{document}

\title{Solvable groups satisfying the two-prime hypothesis II}

\author{James Hamblin \\
        Shippensburg University \\
        Shippensburg, PA 17257 \\
        e-mail: JEHamb@ship.edu \\
        \\
        Mark L. Lewis \\
        Department of Mathematical Sciences \\
        Kent State University \\
        Kent, OH 44242 \\
        e-mail: lewis@math.kent.edu}

\maketitle
\vfill

\begin{abstract}
In this paper, we consider solvable groups that satisfy the two-prime hypothesis.  We prove that if $G$ is such a group and $G$ has no nonabelian nilpotent quotients, then $|\cd G| \le 462,\!515$.  Combining this result with the result from part I, we deduce that if $G$ is any such group, then the same bound holds.

Keywords: character degrees, Finite groups, Solvable groups.

MSC[2010]: 20C15
\end{abstract}

\eject

\section{Introduction}

One general category of problems in the character theory of finite groups involves investigating how the structure of the set of degrees of irreducible characters, denoted $\cd G$, influences the structure of the group.

A well-known result states that if $G$ is solvable and the elements of $\cd G$ are pairwise coprime, then $G$ has at most three distinct character degrees.  There are many possible extensions to this problem.  One was investigated by the second author in \cite{onep}.  He assumed that the group $G$ satisfies the ``one-prime hypothesis,'' that is, if you choose $a,b \in \cd G$ distinct, then $\gcd (a,b)$ is either 1 or a prime.

If $n$ is a positive integer, we will write $\omega (n)$ to denote the total number of prime divisors of $n$, counting multiplicity. As an abbreviation, we will write $\omega (a,b)$ to mean $\omega (\gcd (a,b))$.  Thus, the one-prime hypothesis can be stated as: if $a,b \in \cd G$ with $\omega (a,b) > 1$, then $a = b$.  Under this hypothesis, the second author proved in \cite{numone} that for a solvable group $|\cd G| \leq 9$.

We can state the ``$n$-prime hypothesis'' in a similar way: if $a,b \in \cd G$ with $\omega(a,b) > n$, then $a=b$.  More generally, we say that a set of positive integers $X$ satisfies the $n$-prime hypothesis if for all $a,b \in \cd G$ with $\omega(a,b) > n$, we have $a=b$.  One hopes that for $G$ solvable and satisfying the $n$-prime hypothesis, $|\cd G| \leq f(n)$ for some function $f$.  At this time, we are not able to prove the existence of such a function for all $n$, but we can prove that a bound exists when $n = 2$.

The main theorem of this paper is the following:

\begin{thmA}
Let $G$ be a solvable group satisfying the two-prime hypothesis.  Assume $G$ does not have a nonabelian nilpotent factor group.  Then $|\cd G| \leq 462,\!515$.
\end{thmA}

This note contains work from the first author's dissertation at the University of Wisconsin.  The first part of that dissertation appeared in the published paper \cite{partone}.
%
%

%
The main result of \cite{partone} was the following:

\begin{thmB}
Let $G$ be a solvable group satisfying the two-prime hypothesis. Assume $G$ has a nonabelian nilpotent factor group. Then $|\cd G| \leq 88$.
\end{thmB}

Combining Theorem A with Theorem B, we obtain the main theorem of the first author's dissertation.

\begin{thmC}
Let $G$ be a solvable group satisfying the two-prime hypothesis.  Then $|\cd G| \leq 462,\!515$.
\end{thmC}

We note that $27$ is the largest known value for $|\cd G|$ where $G$ is a solvable group satisfying the two-prime hypothesis.  Obviously, there is a wide gap between the upper bound we have found and the known groups.  The goal in this paper is to prove that a bound exists, and we have made no attempt to find an optimal bound.  However, it seems likely that even if one were to optimize the bound using the arguments in this paper that there would still be a large gap between the bound found here and the known groups, so there is still much room for the bound to be improved.

%

\section{A Few Technical Results} \label{technical}

We begin with a few technical lemmas which will be useful in the cases remaining to be considered.  The following fact is quite useful and will be used several times.  It is an easy application of the Glauberman correspondence (see Theorem (13.24) of \cite{text}).


\begin{lemma} \label{7.1}
Let $G$ act on $A$ by automorphisms with $G$ solvable, $A$ abelian, and $(|G|,|A|) = 1$.  Then the action of $G$ on $A$ and the induced action of $G$ on $\irr A$ are permutation isomorphic.
\end{lemma}

We write $G^\infty$ for the intersection of all $N \nor G$ with $G/N$ nilpotent.  Note that $G^\infty$ is a characteristic subgroup of $G$ and is the smallest normal subgroup of $G$ with a nilpotent factor group.

If $n$ is an integer, we use $\pi (n)$ to denote the set of primes that divide $n$.  If $H$ is a subgroup of $G$, we use $\pi (G:H) = \pi (|G:H|)$.


\begin{lemma} \label{7.2}
Let $G$ be a solvable but not nilpotent group, and assume $G^\infty$ is the unique minimal characteristic subgroup of $G$.  Write $N = G^\infty$.
\begin{enumerate}
\item Then $N$ is an elementary abelian Sylow $q$-subgroup of $G$ for some prime $q$
\item If $G/N$ is abelian, then $G/N$ acts on $N$ with a regular orbit and $|G:N| \in \cd G$
\item Let $n > 1$ be an orbit-size of the action of $G/N$ on $N$.  Then $\pi(n) = \pi(G:N)$.
\end{enumerate}
\end{lemma}

\begin{proof}
First note that since $G$ is solvable, $N$ must be an elementary abelian $q$-group for some prime $q$.  Let $Q \in \syl qG$ and observe that $N \sbs Q$.  We want to prove that $N = Q$.  By the uniqueness of $N$, we have $\oh {q'}G = 1$.  Hence, by the Hall-Higman theorem (``Lemma 1.2.3'') it follows that $\cent GQ \sbs Q$.

Since $G$ is not nilpotent, $Q < G$.  Let $H/N$ be a complement for $Q/N$ in $G/N$.  Because $G/N$ is nilpotent and $N$ is characteristic, it follows that $H$ is a characteristic subgroup of $G$.  Consider the action of $H$ on $Q$.  As $N$ is abelian, we have that $H/N$ acts on $Q$, and this action is coprime.  Hence, we may write $Q = \cent Q H \cdot [Q,H]$.

Because $G/N = H/N \times Q/N$, we see that $H/N$ centralizes $Q/N$.  In particular, this implies that $[Q,H] \sbs N$.  Now since $H$ and $Q$ are characteristic subgroups of $G$, it follows that $\cent QH$ is also a characteristic subgroup of $G$.

Assume $\cent Q H > 1$.  Then by the uniqueness of $N$, we must have $N \sbs \cent Q H$.  Now $[Q,H] \sbs N \sbs \cent Q H$ and therefore $Q = \cent Q H \cdot [Q,H] = \cent Q H$.  However, we now have $H \sbs \cent G Q \sbs Q$, which is a contradiction.

Thus we may assume $\cent Q H = 1$.  Now $Q = \cent Q H \cdot [Q,H] = [Q,H] \sbs N$.  Therefore $Q = N$, as desired.

For (b), observe that $\cent G N \sbs N$, and so $G/N$ acts
faithfully on $N$.  Now, Lemma (2.3) of \cite{partone} implies that $G/N$ has a regular orbit on $N$, as desired.  By Lemma
\ref{7.1}, there must also be a regular orbit of $G/N$ on $\irr N$.  Let $\lambda \in \irr N$ be an element of this orbit.  Then $I_G(\lambda) = N$, and so by the Clifford correspondence (Theorem 6.11 of \cite{text}), $\lambda^G \in \irr G$.  We have $\lambda^G (1) = |G:N| \in \cd G$, which gives (b).

To prove (c), let $1 \not= x \in N$.  Write $C = \cent Gx$ and $n = |G:C|$.  Since $n$ divides $|G:N|$, clearly $\pi (n) \sbs \pi (G:N)$.  For the reverse containment, assume $p$ is a prime dividing $|G:N|$ but that $p$ does not divide $n$.  In particular, this implies that $C/N$ contains a full Sylow $p$-subgroup of $G/N$.

Let $P/N \in \syl p{G/N}$.  Note that since $G/N$ is nilpotent, it follows that $P$ is a characteristic subgroup of $G$.  Now since $P$ centralizes $x$, we have $x \in \zent P$.  In particular, $\zent P$ is a nontrivial characteristic subgroup of $G$, and so $N \sbs \zent P$ by the uniqueness of $N$.  However, we now have $P \sbs \cent GN \sbs N$.  Since $p$ divides $|G:N|$, it must be that $P/N > 1$, which is a contradiction.
\end{proof}

We now consider $G^\infty$ when $G$ is a normal subgroup of some overgroup $\Gamma$.  This lemma and its sequels should be compared with Lemmas 2.2 and 2.3 of \cite{IsLe}.

\begin{lemma}\label{7.3}
Assume $\Gamma$ is solvable and $G \nor \Gamma$ is nonabelian with $G' \sbs \Oh pG$ for all primes $p$. Let $K \nor \Gamma$ be maximal with $K \sbs G$ and $G/K$ nonabelian.  Write $N/K = (G/K)'$. Then $N/K = (G/K)^\infty$ and $N/K$ is the unique minimal characteristic subgroup of $G/K$.
\end{lemma}

\begin{proof}
We first prove that $N/K$ is the unique minimal characteristic
subgroup of $G/K$.  Let $U/K$ be a nontrivial characteristic
subgroup of $G/K$.  Now $U \nor \Gamma$ and $U > K$, so by the way $K$ was chosen we must have $G/U$ abelian.  Then $U/K \sps (G/K)' = N/K$, as desired.

Since $G/K$ is nonabelian and $G' \sbs \Oh p G$ for all primes $p$, it follows that $G/K$ is not nilpotent.  In particular,
$(G/K)^\infty > 1$.  By the uniqueness of $N/K$, we have $N/K \sbs (G/K)^\infty$.  Also, $G/N$ is abelian and in particular $G/N$ is nilpotent.  Thus $(G/K)^\infty \sbs N/K$.  Therefore $N/K = (G/K)^\infty$, as desired.
\end{proof}

We now apply Lemmas \ref{7.2} and \ref{7.3}.

\begin{lemma}\label{7.4}
Assume $\Gamma$ is solvable and $G \nor \Gamma$ is nonabelian
with $G' \sbs \Oh p G$ for all primes $p$.  Let $K \nor \Gamma$ be maximal with $K \sbs G$ and $G/K$ nonabelian.  Write $N/K = (G/K)'$ and $f = |G:N|$. Assume further that $K$ is chosen so that $f$ is minimized.
\begin{enumerate}
\item $N/K$ is an elementary abelian Sylow $q$-subgroup of $G/K$ for some prime $q$
\item Assume $p \not= q$ and $G/\Oh p N$ is nonabelian.  Suppose further that $M \nor \Gamma$ with $\Oh p N \sbs M \sbs G$ and $M$ is maximal with $G/M$ nonabelian.  Then $(G/M)' = N/M$.
\item Assume $G/N$ is cyclic.  Suppose $S \nor \Gamma$ such that $S \sbs N$ and $N/S$ is an abelian $q$-group.  Suppose further that $\cent {N/S} {G/N} = 1$.  Then the action of $G/N$ on $N/S$ is Frobenius.
\end{enumerate}
\end{lemma}

\begin{proof}
By Lemma \ref{7.3}, we have that $N/K = (G/K)^\infty$ is the unique minimal characteristic subgroup of $G/K$.  In particular, $G/K$ is not nilpotent and hence Lemma \ref{7.2} applies to the group $G/K$.  Statement (a) clearly follows from Lemma \ref{7.2}(a).

We now work to prove (b).  Let $M$ be as stated and write $C/M = (G/M)'$.  Now, Lemma \ref{7.3} implies that $C/M = (G/M)^\infty$ is the unique minimal characteristic subgroup of $G/M$.  Thus, $G/M$ is not nilpotent and we may apply Lemma \ref{7.2} to the group $G/M$.

Since $G/N$ is abelian, we have that $G/MN$ is abelian.  Hence $C \sbs MN$.  Now $|MN:M| = |N:N\cap M|$ divides $|N:\Oh p N|$ which is a power of $p$.  By Lemma \ref{7.2}(a) applied to $G/M$, we have that $C/M$ is a Sylow subgroup of $G/M$.  Thus it cannot be that $C < MN$. Hence $C = MN$.

Now, $C = MN \sps N$, so we have $f = |G:N| \geq |G:C| \geq f$, where the final inequality follows by the minimality of $f$.  Thus, $|G:N| = |G:C|$ and since $C \sps N$, we have $C = N$ as desired.

Finally, we prove (c).  Let $1 \not= x \in N/S$ and write $E/S = \cent {G/S} x$.  Note that since we have assumed $\cent {N/S} {G/N} = 1$, it follows that $N \sbs E < G$.  We want to show that $E = N$.  Assume otherwise.  Since $G/N$ is cyclic, note that $E \nor \Gamma$.

Consider the action of $E/N$ on $N/S$.  Now $N/S$ is an abelian $q$-group and $q$ does not divide $|G:N|$.  Thus this action is coprime and we may apply Fitting's Theorem.  Write $A/S = \cent {N/S} {E/N}$ and $B/S = [N/S, E/N]$ so that $N/S = A/S \times B/S$.  Note that since $A$ and $B$ are uniquely determined by subgroups normal in $\Gamma$ we have that $A,B \in \Gamma$.

Now $E/N$ acts trivially on $N/B$, so $N/B \sbs \zent {E/B}$. Since $E/N$ is cyclic it follows that $E/B$ is abelian.  Let $F/B \sbs E/B$ be a complement for $N/B$.  In particular, $F/B$ is a normal Hall subgroup of $E/B$, and since $E \nor \Gamma$ it follows that $F \nor \Gamma$.

Now $E/S = A/S \times F/S$, and so $E/F \cong A/S$ as $G$-modules.  We next want to show that $G$ acts nontrivially on $E/F$.  It suffices to show that $G$ acts nontrivially on $A/S$.  Assume otherwise. Then we have $\cent {G/S} {A/S} = G/S$ and in particular,
\[
1 \not= x \in A/S \sbs \cent {N/S} {G/N} = 1.
\]
This is a contradiction, and hence $G$ acts nontrivially on $E/F$, as desired.  Hence $G/F$ is nonabelian.

Let $L \nor \Gamma$ with $F \sbs L \sbs G$, where $L$ is chosen maximal with $G/L$ nonabelian.  Again by Lemma \ref{7.3}, we may apply Lemma \ref{7.2} to the group $G/L$.  Write $C/L = (G/L)'$.  Since $G/N$ is abelian, $G/NL$ is abelian, so $C \sbs NL$.  Now $S \sbs L \cap N$, and in particular, $|NL:L| = |N:L\cap N|$ divides $|N:S|$, which is a power of $q$.  From Lemma \ref{7.2} (a) applied to $G/L$ we conclude that
$(|NL:C|,|C:L|) = 1$, and thus $C = NL$.  In particular, $C \sps N$.  Now, $f = |G:N| \geq |G:C| \geq f$, where the last inequality follows by the minimality of $f$.  Hence $N = C \sps L \sps F$, which contradicts the choice of $F$.
\end{proof}

We obtain further information in this situation.

\begin{lemma}\label{7.5}
Assume $\Gamma$ is solvable and $G \nor \Gamma$ is not
nilpotent. Let $K \nor \Gamma$ be maximal with $K \sbs G$ and $G/K$ not nilpotent.  Write $N/K = (G/K)^\infty$. Then
\begin{enumerate}
\item $N/K$ is the unique subgroup of $G/K$ minimal with the property of being normal in $\Gamma$
\item $N/K$ is an elementary abelian Sylow $q$-subgroup of $G/K$ for some prime $q$
\item If $n$ is the size of a nontrivial orbit of the action of $G/N$   on $N/K$, then $\pi(n) = \pi(G:N)$.
\end{enumerate}
\end{lemma}

\begin{proof}
For (a), let $K < U \sbs G$ with $U \nor \Gamma$.  Then by choice of $K$, it must be that $G/U$ is nilpotent, and so $U/K \sps (G/K)^\infty = N/K$.  Hence $U \sps N$, as desired.

It follows by (a) that $N/K$ is the unique minimal characteristic subgroup of $G/K$.  Hence (b) and (c) are immediate from Lemma \ref{7.2} (a) and (c), respectively.
\end{proof}

We now consider a cyclic group acting semi-regularly on the basis of a vector space.

\begin{lemma} \label{7.6}
Let $F$ act on $V$, where $F$ is a cyclic group and $V$ is a vector space.  Assume $V$ has a basis which is permuted semi-regularly by $F$.  Write $|F| = f$ and let $s$ be a divisor of $f$.  Then there exists an $F$-orbit in $V$ of size $s$.
\end{lemma}

\begin{proof}
Since $V$ has a basis which is permuted semi-regularly by $F$, we can find a linearly independent set $v_1, v_2, \ldots, v_f \in V$ which is a complete $F$-orbit.  Write $F = \langle x
\rangle$.  Choose notation so that for $1 \leq i < f$ we have $v_i^x = v_{i+1}$ and $v_f^x = v_1$.

Write $y = x^{s}$.  Since $|F:\langle y \rangle| = s$, we will be done if we can construct a vector $w \in V$ with $F_w = \langle y \rangle$. Define
\[
w = v_1 + v_1^y + v_1^{y^2} + \cdots + v_1^{y^{f/s-1}} = v_1 +
v_{1+s} + v_{1+2s} + \cdots + v_{1+f-s}.
\]
Note that since $y^{f/s} = 1$ we have $y \in F_w$.  Conversely, assume $z \in F_w$.  Then since the $v_i$ are linearly independent, $v_1^z$ must be of the form $v_{1+is}$ for some $i$ with $0 \leq i \leq f/s-1$.  In particular, we see that $z = x^{is}$ for some $i$, and hence, $z \in \langle y \rangle$, as desired.
\end{proof}

\section {A unifying theorem}

We now consider a useful generalization of the $n$-prime hypothesis, which we define here and will use several times.

First, we must define some notation.  If $n$ is a positive integer and $\pi$ is a set of primes, write $n_\pi$ for the $\pi$ part of $n$. Also, write $\pi'$ for the complement of $\pi$ in the set of all primes. If $X$ is a set of positive integers, write $X_\pi = \{ x_\pi \mid x\in X \}$.

\medskip
{\bf Definition:}  Let $X$ be a set of positive integers, and let $\pi$ be a set of primes.  Then $X$ satisfies the ``$\hyp n\pi$'' if whenever $a,b \in X$ with $\omega(a,b) > n$, then $\subdash a\pi = \subdash b\pi$. If $\cd G$ satisfies this hypothesis, we will abbreviate by saying $G$ does.  We will refer to the ``$\hyp n{ \{ p \}}$'' as the ``$\hyp np$''.
\medskip

Note that the $\hyp n\emptyset$ is equivalent to the $n$-prime
hypothesis.  Also, if a group satisfies the $\hyp n\pi$ for some set $\pi$ of primes, then there is no information that can be obtained about the $\pi$-parts of elements of $\cd G$.  The question then becomes: if a group satisfies the $\hyp n\pi$, can we bound the number of $\pi'$-parts of character degrees in terms of $n$?  Here are two results of this type that were proved in \cite{partone}.

\begin{thmD}
Let $G$ be solvable and assume $G$ satisfies the $\hyp 1\pi$. Then
$|\subdash {\cd G}\pi| \leq \frac{3}{2}|\pi|^2 + \frac{19}{2}|\pi| + 18$.
\end{thmD}

\begin{thmE}
Let $G$ be solvable and nonabelian, and assume $G$ satisfies the $\hyp n\pi$.  Let $K \nor G$ be maximal with $G/K$ nonabelian. Write $N/K = (G/K)'$ and assume $N$ is nilpotent.  Then
\[
|\subdash {\cd G}\pi| \leq (1 + 2^{2n})^3.
\]
\end{thmE}

The following theorem unifies many of the remaining cases, and
allows us to apply the results of Section 6 of \cite{partone}.


\begin{theorem}\label{7.7}
Let $\Gamma$ be a finite solvable group.  Assume:
\begin{enumerate}
\item $\Gamma$ satisfies the two-prime hypothesis
\item $K, N, G$ are normal subgroups of $\Gamma$, where $G/K$ is a Frobenius group with Frobenius kernel $N/K$
\item $G/N$ is cyclic of order $f$ and $N/K$ is an elementary abelian $q$-group with $|N:K| \leq q^2$
\item $|\pi(f)| \geq 8$
\item $G/\Oh p N$ is abelian for all primes $p \not= q$
\item For every $S \nor \Gamma$ with $\Oh q N \sbs S \sbs N$ and $\cent {N/S} {G/S} = N/S$, we have that the action of $G/N$ on $N/S$ is Frobenius
\item $q$ does not divide $|\Gamma: G|$.
\end{enumerate}
Then $N$ is nilpotent.
\end{theorem}

\newcounter{stepcount} \setcounter{stepcount}{1}
\newcommand{\step}{\medskip
\noindent{\bf Step \arabic{stepcount}}\stepcounter{stepcount}. }

\begin{proof}
Assume $N$ is not nilpotent.  We will take a normal subgroup of $\Gamma$, contained in $N$, which is chosen maximal with a non-nilpotent factor.  Using results from Section \ref{technical}, we work towards a contradiction.

Let $M \sbs N$, where $M \nor \Gamma$ and $M$ is chosen maximal such that $N/M$ is not nilpotent.  Write $L/M = (N/M)^\infty$.  Now, Lemma \ref{7.5} applies to the group $N/M$.  From Lemma \ref{7.5} (a), we have that $L/M$ is the unique subgroup of $N/M$ which is minimal with the property of being normal in $\Gamma$.  Also, by Lemma \ref{7.5} (b), we conclude that $L/M$ is an elementary abelian Sylow $s$-subgroup of $N/M$ for some prime $s$.  Finally, by Lemma \ref{7.5} (c), we have that if $1 \not= x \in L/M$, then $\pi(N/M:\cent{N/M}x) = \pi(N:L)$. Since $N/L$ is nilpotent, we may write $N/L = P/L \times Q/L$, where $P/L = \oh {q'} {N/L}$ and $Q/L = \oh q {N/L}$.


\step $G/Q$ is abelian.

\begin{proof}
By assumption (5), it follows that $G' \sbs \Oh p N$ is abelian for all primes $p\not=q$.  Since $N/Q$ is nilpotent, we have
\[
\bigcap_{p\in \pi(N:Q)} \Oh p N \sbs Q.
\]
Now $q$ does not divide $|N:Q|$, and hence $G' \sbs Q$.
\end{proof}


\step Assume $L < X \sbs N$ with $X \nor \Gamma$.  Then $\cent {L/M}{X/L} = 1$.

\begin{proof}
Assume otherwise.  Write $Z = L/M \cap \zent{X/M}$. Then $Z > 1$.  Since $L,X \nor \Gamma$, it follows that $Z \nor \Gamma$.  By the uniqueness of $L$, we must have $Z \sps L$.

Now $L/M$ is a normal, central Sylow $s$-subgroup of $Z/M$. Hence, if we write $U/M=\oh {s'} {Z/M}$, it follows that $Z/M = L/M \times U/M$.  On the other hand, $U \nor \Gamma$, and since $X > L$, we must have $U > M$.  However, $U \not\sps L$, which contradicts the uniqueness of $L$.
\end{proof}


\step Assume $a \in \cd G$ with $\omega(a,f) \geq 3$. Then $q$ does not divide $a$.

\begin{proof}
Since $\Gamma$ satisfies the two-prime hypothesis, it follows by Lemma (3.6) of \cite{partone} that $G$ satisfies the $(2,\pi(\Gamma:G))$-hypothesis. Since $\omega(a,f) \geq 3$, the $(2,\pi(\Gamma:G))$-hypothesis implies that $\subdash a{\pi(\Gamma:G)} = \subdash f{\pi(\Gamma:G)}$. Now, since $q \notin \pi(\Gamma:G)$, this implies $a_q = f_q = 1$, as desired.
\end{proof}


\step Let $\lambda \in \irr {L/M}$.  Then $\lambda$ extends to
$I_G(\lambda)$.

\begin{proof}
Write $T = I_G(\lambda)$ and $S = I_N(\lambda)$.  Now, since $L/M$ is a Sylow $s$-subgroup of $N/M$, we have that $|S:L|$ is coprime to $|L:M|$.  In particular, if we write $o(\lambda)$ for the order of $\lambda$ in the group $\lin {L/M}$, it must be that $|S:L|$ is coprime to $o(\lambda)$.  Then it follows by Corollary (6.28) of \cite{text} that $\lambda$ has a unique extension $\hat\lambda \in \irr S$ with $o (\hat\lambda) = o (\lambda)$.

Since $\lambda$ uniquely determines $\hat\lambda$, and certainly $\hat\lambda$ determines $\lambda$, we have that $I_G(\hat\lambda) = T$.  Now $T/S = T/(T\cap N)$ is isomorphic to $TN/N$, which is a subgroup of $G/N$.  Hence $T/S$ is cyclic.  Thus by Corollary (11.22) of \cite{text}, we have that $\hat\lambda$ extends to $T$.  Since $\hat\lambda$ is an extension of $\lambda$, we are done.
\end{proof}


\step Let $\lambda \in \irr {L/M}$ be nonprincipal.  Then
$I_Q(\lambda) \sbs K$.

\begin{proof}
Assume otherwise.  Then in particular, $I_Q(\lambda) \cap K <
I_Q(\lambda)$.

We first work to show that $|I_G(\lambda):I_N(\lambda)|$ is an
irreducible character degree of $I_G(\lambda)/L$.  Note that this is trivial if $I_G(\lambda) = I_N(\lambda)$.  Hence, we may assume that $I_G(\lambda) > I_N(\lambda)$.

We have $I_Q(\lambda) \sbs I_N(\lambda)$, so
$I_N(\lambda) \not\sbs K$.  In particular, $I_N(\lambda) >
I_K(\lambda)$.  Since $G/K$ is a Frobenius group with kernel $N/K$, it follows that $I_G(\lambda)/I_K(\lambda)$ is a Frobenius group with kernel $I_N(\lambda)/I_K(\lambda)$.  Therefore, $|I_G(\lambda):I_N(\lambda)|$ is an irreducible character degree of $I_G(\lambda)/L$, as desired.  Write $g =
|I_G(\lambda):I_N(\lambda)|$.

Next, we prove that $|G:I_N(\lambda)| \in \cd G$. By Step 4, we have that $\lambda$ extends to $I_G(\lambda)$. Now, Gallagher's Theorem (Corollary 6.17 of \cite{text}) implies that since $g \in \cd {I_G(\lambda)/L}$ we may choose $\psi \in \irr
{I_G(\lambda) \mid \lambda}$ with $\psi(1) = g$. By the Clifford correspondence (Theorem 6.11 of \cite{text}), we have $\psi^G \in \irr G$. Thus
\[
\psi(1) = |G:I_G(\lambda)| g = |G:I_N(\lambda)| \in \cd G,
\]
as desired.

We now show that $q$ divides $|N:I_N(\lambda)|$. Since $N/L$ is nilpotent and $Q/L$ is a Sylow $q$-subgroup of $N/L$, it suffices to show that $Q > I_Q(\lambda)$.  Assume otherwise. Then in the action of $Q/L$ on $\irr {L/M}$ there is a nontrivial fixed point.  By Lemma \ref{7.1}, it follows that the action of $Q/L$ on $L/M$ has a nontrivial fixed point.  This contradicts Step 2 since $Q \nor \Gamma$.

Finally, we have $|G:N|\cdot |N:I_N(\lambda)| = f\cdot
|N:I_N(\lambda)| \in \cd G$.  This character degree is divisible by both $f$ and $q$, which contradicts Step 3.  The contradiction arose from our assumption that $I_Q(\lambda) \not\sbs K$, and this completes the proof of the step.
\end{proof}


\step
Assume $\sigma \sbs \pi(f)$ and let $F/L \nor G/L$ be a
Frobenius group.  Suppose $R/L \sbs Q/L$ is the kernel of $F/L$ and assume $R \nor \Gamma$.  Assume $|F:R| = f_\sigma$.  Then $|\sigma| \leq 3$.

\begin{proof} Assume $|\sigma| \geq 4$. Since $R \nor \Gamma$, Step 2 implies that $R/L$ acts nontrivially on $L/M$.  Hence, Theorem 15.16 of \cite{text} applies. We obtain a basis for $L/M$ which is permuted semi-regularly by a Frobenius complement $H/L$ of $F/L$.  Viewing $L/M$ as a vector space over ${\bf F}_s$, we may appeal to Lemma \ref{7.6}.

Let $p \in \sigma$, and write $a = f_\sigma/p$.  By Lemma \ref{7.6}, there exists an $H/L$-orbit in $L/M$ of size $a$.  In view of Lemma \ref{7.1}, we must also have $\alpha \in \irr {L/M}$ which lies in an $H/L$-orbit of size $a$.  Recall that $R/L \sbs Q/L$ is a $q$-group. Now, Lemma (2.5) of \cite{partone} implies that $aq^u = |F:I_F(\alpha)|$ for some $u \geq 0$.  By Step 2, we cannot have $R \sbs I_F(\lambda)$.  Thus, we must have $u > 0$ since $R \nor \Gamma$.

By Step 4, we have that $\alpha$ extends to $I_F(\alpha)$. Using the Clifford Correspondence (Theorem 6.11 of \cite{text}), we obtain $aq^u \in \cd F$. Now, consider $b \in \cdover G {aq^u} F$. Since $F \nor G$, we have $aq^u$ divides $b$. Now $\omega(b,f) \geq \omega(aq^u, f) \geq |\sigma| - 1 \geq 3$. But $q$ divides $b$, which contradicts Step 3.
\end{proof}


\step We cannot have both $G/P$ a Frobenius group (with kernel
$N/P$) and $1 \not= \lambda \in \irr {L/M}$ with $I_Q(\lambda) = L$.

\begin{proof}
First, we claim that $|\pi(P:L) \cap \pi(f)| \leq 2$.  Assume
this is not true, and let $1 \not= \alpha \in \irr {L/M}$.  By Lemma \ref{7.5} (c), $\alpha$ lies in a $N/L$-orbit of some size $n$ with $\pi(n) = \pi(N:L)$.  Now, we are assuming $|\pi(P:L)\cap \pi(f)| \geq 3$, and so $|\pi(N:L)\cap \pi(f)| \geq 3$.  In particular, $\omega(n,f) \geq 3$.  By Step 3, $q$ does not divide $n$.  However, $q \in \pi(N:L) = \pi(n)$, so $q$ does divide $n$. This is a contradiction.  Hence, $|\pi (P:L) \cap \pi (f)| \le 2$.

Assume $G/P$ is a Frobenius group with kernel $N/P$, and assume there exists $1 \not= \lambda \in \irr {L/M}$ with $I_Q(\lambda) = L$.  Write $T = I_G(\lambda)$.  Now, $\lambda$ extends to $T$ by Step 4, so $|G:T| \in \cd G$ by Clifford Correspondence.  Since $I_Q(\lambda) = L$, we have $|Q:L|$ divides $|G:T|$.  It follows by Step 3 that $\omega(G:NT)\leq\omega(f, |G:T|) \leq 2$.

Let $\sigma = \pi(f) \setminus ( \pi(P:L) \cup \pi(G:TN) )$.  Then, $|\sigma| \geq |\pi(f)| - 4$ by our arguments above. We obtain $|\sigma| \geq 4$ by assumption (4).

Note that $\sigma \sbs \pi(TN:N) = \pi(TQ:(TQ\cap N))$, and so, $|TQ:Q|$ is divisible by all the primes in $\sigma$.  Let $S = \oh \sigma {TQ/Q}$.  Now, $Q \sbs S \cap N \sbs N$, and so, $|(S \cap N):Q|$ divides $|N:Q| = |P:L|$.  However, $S/Q$ is a $\sigma$-group and $\sigma \cap \pi(P:L) = \emptyset$ by the way $\sigma$ was defined.  Therefore, $S\cap N = Q$.

We have $SN > N$, and hence, $SN/P \sbs G/P$ is a Frobenius group with kernel $N/P$.  In particular, $SN/P \cong S/L$, and so $S/L$ is a Frobenius group with kernel $Q/L$.  We now have a contradiction with Step 6.  Since $G/Q$ is abelian, it follows that $S\nor G$.  Also, $|S:Q| = |TN:N| = f_\sigma$.  Hence all the hypothesis of Step 6 are satisfied, but we have $|\sigma| \geq 4$, which is the desired contradiction.
\end{proof}


\step Assume $1 \ne \lambda \in \irr {L/M}$.  Then $I_Q(\lambda) < K \cap Q$.

\begin{proof}
Assume otherwise.  By Step 4, $I_Q(\lambda) \sbs K \cap Q$.  Hence, $I_Q(\lambda) = K \cap Q$, so $K \cap Q \nor \Gamma$. Suppose $K \cap Q > L$.  Recall that by Step 2 $(K\cap Q)/L$ fixes no nontrivial element of $L/M$.  By Lemma \ref{7.1}, $(K\cap Q)/L$ fixes no nonprincipal character of $\irr {L/M}$.  However, $(K\cap Q)/L$ fixes $\lambda$, which is a contradiction.  We deduce that $K \cap Q = L$.  In particular, $K = P$ and $G/P$ is a Frobenius group with kernel $N/P$.  We also see that $I_Q(\lambda) = K \cap Q = L$.  This contradicts Step 7.
\end{proof}


\step All nontrivial orbit-sizes of the action of $Q/L$ on $L/M$ are divisible by $q^2$.  This action has at most two nontrivial orbit-sizes.  If there are two orbit sizes, then $|N:K| = q$ and one orbit-size must be $q^2$.

\begin{proof}
By Lemma \ref{7.1}, it will suffice to consider the action of $Q/L$ on $L/M$. Let $1 \not= \lambda \in \irr {L/M}$.  Applying Step 8, $I_Q(\lambda) < K\cap Q$, so $|N:K|q$ divides
$|Q:I_Q(\lambda)|$. In particular, $q^2$ divides $|Q:I_Q(\lambda)|$, and this proves the first statement.

Since $Q/L$ is a normal Sylow $q$-subgroup of $G/L$, we have $|Q:I_Q(\lambda)| = |G:I_G(\lambda)|_q$.  Now, $\lambda$ extends to $I_G(\lambda)$ by Step 4, and thus, $|Q:I_Q(\lambda)|$ is the $q$-part of some element of $\cd G$ by the Clifford Correspondence (Theorem 6.11 of \cite{text}).

Using Lemma (3.6) of \cite{partone}, $G$ satisfies the $(2,\pi(\Gamma:G))$-hypothesis.  Note that $q \notin \pi(\Gamma:G)$ by assumption (7).  Thus, Lemma (3.4) of \cite{partone} implies that $\cd G_q$ satisfies the two-prime hypothesis.  In particular, there are at most two elements of this set divisible by $q^2$, as desired.  Furthermore, only one element of $\cd G_q$ can be divisible by $q^3$.  Finally, since $|N:K|q$ divides each orbit size, the only way there can be two orbit sizes is if $|N:K| = q$ and one of the orbits has size $q^2$.
\end{proof}


\step The action of $Q/L$ on $L/M$ has exactly two nontrivial
orbit-sizes.

\begin{proof}
Assume this is false.  The action of $Q/L$ on $L/M$ is nontrivial by Step 2.  By Step 9, the action has at most two orbit sizes.  It must be that $Q/L$ acts $\frac {1}{2}$-transitively on $L/M$.  Since $|\pi(f)| \geq 8$ and $G/K$ is a Frobenius group, clearly we have $q > 2$.  Hence, $Q/L \cong N/P$ is cyclic by Lemma 3 from \cite{ispa}.

Consider the action of $G/N$ on $N/P$.  This is a coprime action on an abelian group, so we may apply Fitting's theorem.  This gives a direct product decomposition of a cyclic $q$-group, so one of the factors must be trivial since there is a unique subgroup of order $q$.  This action is not trivial, hence, $\cent {N/P}{G/N} = 1$. Since $G/N$ is cyclic, we conclude that $G/P$ must be a Frobenius group with kernel $N/P$.

Every subgroup of $Q/L$ is characteristic in $Q/L$, hence normal in $\Gamma/L$.  In view of Step 2, if $1 \not= \lambda \in \irr {L/M}$, then $I_Q(\lambda) = L$.  This contradicts Step 7.
\end{proof}

Combining Steps 9 and 10, we obtain $|N:K| = q$.


\step $P = L$ and $Q = N$.

\begin{proof}
By Step 10, there exist characters $\alpha, \beta \in \irr {L/M}$ which lie in nontrivial $Q/L$-orbits of different sizes.  Write $|Q:I_Q(\alpha)| = q^a$ and $|Q:I_Q(\beta)| = q^b$, and Step 9 implies that one of $a$ and $b$ must equal 2.  Without loss we may assume $2 = a < b$.

Assume $P > L$, and let $p \in \pi(P:L)$.  Then $p$ divides
$|N:I_N(\lambda)|$ for all $1 \not= \lambda \in \irr {L/M}$ by
Lemma \ref{7.5} (c). In particular, $p$ divides both $|N:I_N(\alpha)|$ and $|N:I_N(\beta)|$. Since $\alpha$ and $\beta$ extend to their inertia groups in $G$, we obtain $x,y \in \cd G$ with $x_q = q^a$, $y_q = q^b$, and both $x$ and $y$ are divisible by $p$.  Therefore, $\omega(x,y) \geq 3$.

From hypothesis (1), we know that $\Gamma$ satisfies the
two-prime hypothesis.  Thus, $G$ satisfies the $(2,\pi(\Gamma:G))$-hypothesis by Lemma (3.6) of \cite{partone}.  This implies $x_{\pi(\Gamma:G)'} = y_{\pi(\Gamma:G)'}$. Since $q \notin \pi(\Gamma:G)$, we have $q^a = x_q = y_q = q^b$, which is a contradiction.
\end{proof}


\step $K/L$ has an orbit of size $q$ on $L/M$.

\begin{proof}
By Steps 9, 10, and 11, it follows that $N/L$ has an orbit of size $q^2$ on $L/M$.  From Lemma \ref{7.1} we also have that $N/L$ has an orbit of size $q^2$ on $\irr {L/M}$. Thus, we may let $1 \not= \lambda \in \irr {L/M}$ with $|N:I_N(\lambda)| = q^2$.  By Step 8, $I_N(\lambda) < K$.  We obtain
\[
|K:I_K(\lambda)| = \frac{|N:I_K(\lambda)|}{|N:K|} = q^2/q = q,
\]
where the second equality follows by the observation before Step 11.
\end{proof}


\step $K/L$ is elementary abelian.

\begin{proof}
Write $F/L = \Phi(K/L)$.  Since $K/L$ is a $q$-group, we have $K/F$ is elementary abelian, and it suffices to show that $F = L$.

By Step 12 and Lemma \ref{7.1}, let $1 \not= x \in {L/M}$ such that $\cent {K/L} x$ has index $q$ in $K$.  In particular, $\cent {K/L} x$ is a maximal subgroup of $K/L$, and thus $F/L \sbs \cent {K/L} x$.

Now $F/L$ is a characteristic subgroup of $K/L$, and hence, $F$ is normal in $\Gamma$.  Also, since $F \sbs \cent {K/L} x$, we have by Step 2 that $F=L$, as desired.
\end{proof}


\step $K/L$ has rank 2.

\begin{proof}
Since $N/L$ has an orbit on $L/M$ of size divisible by $q^3$, it must be that $|N/L| \geq q^3$.  Hence, $|K/L| \geq q^2$. Note that $K/L$ acts faithfully on $L/M$ by Step 2 since $\cent {K/M} {L/M} \nor \Gamma$.

Assume the rank of $K/L$ is 3 or greater.  Then by Lemma 2.6
from \cite{riedl}, we obtain at least three distinct nontrivial orbit-sizes of the action of $K/L$ on $L/M$, and we have three nontrivial orbit-sizes of the action of $N/L$ on $L/M$ by Step 8.  This contradicts Step 10.
\end{proof}


\step $N/L$ is nonabelian.

\begin{proof}
Assume otherwise.  Apply Fitting's lemma to the action of $G/N$ on $N/L$ to obtain $N/L = A/L \times B/L$ where $A/L = \cent {N/L}{G/N}$.  Since ``fixed points come from fixed points'' in a coprime action, $\cent {N/A} {G/N} = 1$.  By assumption (6), $G/A$ is a Frobenius group.

Now, $G/B$ is abelian, so let $E/B$ be a complement for $N/B$ in $G/B$.  Then $E/L \cong G/A$ is a Frobenius group with kernel $B/L$.  This contradicts Step 6.
\end{proof}


We now work for the final contradiction.
By Steps 14 and 15, it follows that  $N/L$ is extraspecial of order $q^3$.  Write $Z/L = \zent {N/L} = (N/L)'$.  Observe that $N/Z$ is a ${\bf F}_q[\Gamma/N]$-module and $K/Z$ is a submodule.  Also, $q$ does not divide $|\Gamma:N|$ by assumptions (2), (3), and (7).  By Maschke's Theorem, we may let $K_0/Z \nor \Gamma/Z$ be a complement for $K/Z$ in $N/Z$.

First, assume $G/K_0$ is abelian.  Let $F/K_0 \nor \Gamma/K_0$ be a complement for $N/K_0$ in $G/K_0$, and $F/Z \cong G/K$ is a Frobenius group. Consider the action of $F/K_0$ on $K_0/L$.
Certainly this action is nontrivial since $F/K_0$ acts Frobeniusly on $K_0/Z$.  Applying Fitting's lemma and the fact that $Z/L$ is the only nontrivial subgroup of $K_0/L$ which is normal in $\Gamma$, we have $\cent {K_0/L}{F/K_0} = 1$.  In particular, since $F/K_0$ is cyclic, $\cent {K_0/L}{E/K_0} = 1$ for all subgroups $E$ with $K_0 < E \le F$.

Consider $1 \ne x \in K_0/L$, and write $E/L = \cent {F/L} x$.  By the previous paragraph, $E = K_0$ and thus, $F/L$ is a Frobenius group with kernel $K_0/L$.  This contradicts Step 6.

We now have $G/K_0$ is nonabelian.  In particular,
$\cent {N/K_0} {G/N} < N/K_0$, and so, $\cent {N/K_0} {G/N} = 1$.  By assumption (6), $G/K_0$ is a Frobenius group.

We can replace $K$ by $K_0$.  Let $1 \ne \lambda \in \irr
{L/M}$.  Then Step 8 implies that $I_N(\lambda) < K_0$. Hence,
$I_N(\lambda) \sbs K\cap K_0 = Z$.  Applying Step 2, we cannot have $I_N(\lambda) = Z$.  Thus, $I_N(\lambda) < Z$, and hence,
$I_N(\lambda) = L$.  We deduce that $N/M$ is a Frobenius group with kernel $L/M$.  The Frobenius complement of $N/M$ is an extraspecial $q$-group, which is impossible by 12.6.15 of \cite{scott}.  
This is our final contradiction.
\end{proof}

\section{The Large Frobenius Case}

Let $G$ be a group and let $N \nor G$.  We write $\irr {G
\mid N} = \{ \chi \in \irr G \mid \ker\chi \not\sbs
N \}$ and $\cd {G \mid N} = \{  \chi (1) \mid \chi \in \irr {G \mid N} \}$.

We can associate a graph with $\cd G$.  We define the graph ${\scr G} (G)$ to be the graph whose vertex set is $\cd G \setminus \{ 1 \}$.  There is an edge between $a$ and $b$ if $\gcd (a,b) > 1$.  It has been proved  that $\scr G (G)$ has at most two connected components when $G$ is solvable (see Theorem 30.2 of \cite{huptext} or Theorem 18.4 of \cite{MaWo}).


\begin{theorem}\label{8.1}
Assume $G$ is solvable satisfying the two-prime hypothesis. Let $L \nor G$ and suppose that $G/L$ is a Frobenius group with kernel $M/L$.  Assume that $L$ is chosen among all such normal subgroups of $G$ so that $|G:M|$ is minimized.  Suppose $G/M$ is cyclic and that $\omega (G:M) > 2$.  Assume further that either $M/L$ is elementary abelian of rank $\geq 3$ or $M/L$ is a direct product of two elementary abelian groups for different primes.  Then $|\cd G| \leq 210$.
\end{theorem}

\begin{proof}
Write $f = |G:M|$ and let $\psi \in \irr M$.  Then by Lemma (2.14) of \cite{partone}, either $f\psi(1) \in \cd G$ or $\V \psi \sbs L$. Let $\theta \in \irr L$ be a constituent of $\psi_L$.

Assume $\theta$ is nonlinear.  Then in particular $\psi \in \irr {M \mid L'}$ is nonlinear. By the two-prime hypothesis in $G$, we cannot have both $f, f\psi(1) \in \cd G$ since we have assumed $\omega(f) > 2$.  Hence, $\V \psi \sbs L$, and consequently $|M:L|$ divides $(\psi(1)/\theta(1))^2$.

If $|M/L| = p^a$ for some prime $p$ and $a \geq 3$, then define $m = p^2$.  If $|M/L| = p^a q^b$ for distinct primes $p, q$ and $a, b > 0$, define $m = pq$.  For all characters $\psi \in \irr {M \mid L'}$, each irreducible constituent of $\psi_L$ must be nonlinear.  It follows from the previous paragraph that $m$ divides $\psi(1)/\theta(1)$ for all constituents $\theta \in \irr L$ of $\psi_L$.

Let $X$ be a connected component of $\scr G (L)$.  Let $a, b \in X$ with $\gcd (a,b) > 1$.  Let $\theta \in \irr L$
with $\theta(1) = a$, and let $\psi \in \irr {M \mid \theta}$.  Our previous arguments show that $ma$ divides $\psi(1)$, and so, $ma$ divides all elements of $\cdover GaL$. Similarly, $mb$ divides all elements of $\cdover GbL$.  If we write $d = \gcd (a,b) > 1$, then we obtain $\omega (md) \geq 3$, and $md$ divides all elements of $\cdover G{\{a,b\}}L$.  Applying the two-prime hypothesis in $G$, we conclude that $|\cdover G{\{ a,b \}}L| = 1$.  By induction, we have $|\cdover GXL| = 1$.

Recall that $\scr G(L)$ has at most two connected components other than $\{ 1 \}$.  Applying the previous paragraph, $|\cd {G \mid L'}| \leq 2$.

Now consider $G/L'$.  By Theorem (2.10) of \cite{partone},
\[
|\cd {G/L'}| \leq 2+2+(1+2^4)(3)(4) = 208.
\]
Thus,
\[
|\cd G| \leq |\cd {G|L'}| + |\cd {G/L'}| \leq 2+208 = 210.
\]
\end{proof}

In this section, we will consider the following case.

\medskip
{\bf Hypothesis (``$f$ Large'')}. Assume $G$ is a solvable group satisfying the two-prime hypothesis. Suppose $G' \sbs \Oh p G$ for all primes $p$, and let $K \nor G$ maximal with $G/K$ nonabelian.  Assume $G/K$ is a Frobenius group with kernel $N/K$ which is an elementary abelian $q$-group.  Write $f = |N:K|$, and assume $K$ is chosen so that $f$ is minimized.  Finally, suppose that $|\pi(f)| \geq 8$.


\begin{lemma} \label{8.2}
Assume the $f$ Large Hypothesis. Let $p \not= q$ be prime and
suppose $G/\Oh p N$ is nonabelian. Then $|\cd G| \leq 210$.
\end{lemma}

\begin{proof}
Let $\Oh p N \sbs M \nor G$ where $M$ is chosen maximal with $G/M$ nonabelian.  By Lemma \ref{7.3} (a), we have $(G/M)' = N/M$.  Since we have assumed that $G$ has no nonabelian nilpotent factor groups, it must be that $G/M$ is not nilpotent.  Lemma 12.3 of\cite{text} implies that $G/M$ is a Frobenius group with kernel $N/M$.

Write $D = M \cap K$.  Since $p \not= q$, the indices $|N:K|$ and $|N:M|$ are coprime, and it follows that $N/D = K/D \times M/D$.  Because $G/N$ acts Frobeniusly on both $N/M$ and $N/K$, it must be that $G/D$ is a Frobenius group with kernel $N/D$.  Now, $N/D$ is a direct product of elementary abelian groups for the primes $p$ and $q$.  Hence, Theorem \ref{8.1} applies, and $|\cd G| \leq 210$.
\end{proof}

We are now able to bound $|\cd G|$ when we are in the $f$ Large Hypothesis.

\begin{theorem} \label{8.3}
Assume the $f$ Large Hypothesis.  Then $|\cd G| \leq 4,\!913$.
\end{theorem}

\begin{proof}
To obtain the conclusion, we will apply Theorem \ref{7.7} to $\Gamma = G$.  To do this, we must show that the hypotheses of Theorem \ref{7.7} hold.

We have assumed that $G$ satisfies the two-prime hypothesis, so condition (1) is satisfied.  Also, $G/K$ is a Frobenius group with kernel $N/K$.  Hence, we have condition (2).  Furthermore, we assumed that $|\pi(f)| \geq 8$, and so condition (4) is satisfied.

Lemma 12.3 of \cite{text} implies by the choice of $K$ that $G/N$ is cyclic, and $N/K$ is an elementary $q$-group. If
$|N:K| \geq q^3$, then Theorem \ref{8.1} implies that $|\cd G| \leq 210$. Thus, we may assume $|N:K| \leq q^2$, which gives condition (3).

For condition (5), let $p \not= q$ be prime and assume $G/\Oh p N$ is nonabelian.  In this case, we use Lemma \ref{8.2} to obtain $|\cd G| \leq 210$. Therefore, we may assume that for all primes $p \not= q$ that $G/\Oh p N$ is abelian.  Hence, condition (5) is satisfied.

We have that $G$ is solvable and nonabelian with $G' \sbs \Oh p G$ for all primes $p$.  Also, $K$ has been chosen maximal with $G/K$ nonabelian and such that $f$ is minimized. Finally, $G/N$ is cyclic. Thus, Lemma \ref{7.3} (b) applies and gives us condition (6).

For condition (7), note that we are taking $\Gamma = G$, and that $q$ does not divide $|\Gamma:G| = 1$.   Therefore, Lemma \ref{7.5} applies, and implies that $N$ is nilpotent.  Now, $G$ meets the hypotheses of Theorem (6.3) of \cite{partone} with $n=2$, and applying that theorem, we conclude that $|\cd G| \leq (1 + 2^4)^3 = 4,\!913$.
\end{proof}

\section{The Small Frobenius Case, part 1} \label{Chapter 9}

In the final three sections, we consider the following case.

\medskip
{\bf Hypothesis (``$f$ Small'')}.  Assume $G$ is a solvable group satisfying the two-prime hypothesis, assume $K \nor G$ is maximal with $G/K$ nonabelian, assume $G/K$ is a Frobenius group with kernel $N/K$, and assume $K$ is chosen such that $f = |G:N|$ is minimal. Write $|N:K| = q^e$, and suppose $|\pi(f)| \leq 7$.


\begin{lemma}\label{9.1}
Let $G$ be a solvable group, and let $S \snor G$. Also, let $\psi \in \irr S$ and $\chi \in \irr{G \mid \psi}$.  Then $\chi(1)/\psi(1)$ divides $|G:S|$.  In particular, if $\pi$ is some set of primes with $\pi(G:S) \sbs \pi$, then $\subdash {\cd G}\pi = \subdash {\cd S}\pi$.
\end{lemma}

\begin{proof}
The first statement follows easily by induction applied to Corollary (11.29) of \cite{text}.  For the second statement, fix $a \in \cd G$ and $\chi \in \irr G$ with $\chi(1) = a$.  Let $\psi \in \irr S$ be a constituent of $\chi_S$ and write $b = \psi(1)$.  The first statement implies that $a/b$ is a $\pi$-number, and hence $\subdash a\pi = \subdash b\pi$. Thus, $\subdash {\cd G}\pi \sbs \subdash {\cd S}\pi$.

Conversely, fix $b \in \cd S$ and $\psi \in\irr S$ with $\psi(1) = b$.  Let $\chi \in \irr G$ be a constituent of $\psi^G$ and write $a = \chi(1)$.  Once again we have $a/b$ is a $\pi$-number and hence $\subdash a\pi = \subdash b\pi$.  Therefore, $\subdash {\cd S}\pi = \subdash {\cd G}\pi$.
\end{proof}

Under the assumptions of the $f$ Small Hypothesis, we would like to consider the cases where $K$ is abelian and $K$ is nonabelian.  We first have the abelian case.


\begin{theorem} \label{9.2}
Assume the $f$ Small Hypothesis.  Assume $K$ is abelian. Then $|\cd G| \leq 208$.
\end{theorem}

\begin{proof} Observe that $G$ satisfies the hypotheses of Theorem (2.10) of \cite{partone}.  By that theorem, we have $|\cd G| \le 2 + 2 + (1 + 2^4) \cdot 3 \cdot 4 = 208$.
\end{proof}


This definition can be found in Section 5 of \cite{partone}.

\medskip
{\bf Definition:} Let $n \geq -1$ and $k\geq 2$ be integers and let $\pi$ be a set of primes.  A group $G$ satisfies the $(n, k, \pi)$-hypothesis if for all $x_1, \ldots, x_k \in \cd G$ with $\omega(x_1, \ldots, x_k) > n$ there exist $1 \leq i < j \leq k$ with $\subdash {(x_i)}{\pi} = \subdash {(x_j)}\pi$.
\medskip

Let $\pi$ be a set of primes.  We define $\pi^*$ to be the set of all $\pi$-numbers.  If $n$ is a positive integer, we set $\pi^n = \{ a \in \pi^* \mid \omega (a) = n \}$ and $\pi^{\le n} = \{ a \in \pi^* \mid \omega (a) \le n \}$.  Observe that $\pi^{\le n} = \bigcup_{i=0}^n \pi^i$.  In Lemma 2.7 of \cite{partone}, the following was proved.

\begin{lemma}\label{2.7}
If $\pi$ is a set of primes, then $\displaystyle |\pi^n| = {{|\pi| + n - 1} \choose n}$.
\end{lemma}

In Lemma 2.8 in \cite{partone}, the following result was proved.

\begin{lemma}\label{2.8}
Let $X$ be a set of integers that satisfies the $n$ prime hypothesis.  Let $\pi$ be a set of primes so that $|\subdash X{\pi}| \le m$.  Then $|X| \le |\pi^{n+1}| + m |\pi^{\le n}|$. In particular, $|X| \le |\pi (X)^{n+1}|$.
\end{lemma}


\begin{theorem} \label{9.3}
Assume the $f$ Small Hypothesis and write $\pi = \pi(G:K)$.  Suppose that $K$ is nonabelian and let $L \nor K$ be maximal such that $K/L$ is nonabelian.  Assume $|\pi(K:L) \setminus \pi| \leq 8$. Then $|\cd G| \leq 27,\!540$.
\end{theorem}

\begin{proof}
Since $|\pi(f)| \leq 7$, we have $|\pi| \leq 8$.  Now,
\[
|\pi(G:L)| = |\pi| + |\pi(G:L) \setminus \pi| = |\pi| + |\pi(K:L) \setminus \pi| \leq 8 + 8 = 16.
\]

First, assume $L$ is abelian.  Then since $L \snor G$, we know by Lemma \ref{9.1} that every irreducible character degree of $G$ divides $|G:L|$.  Now, $\cd G$ satisfies the two-prime hypothesis and $\cd G_{\pi (G:L)'} = \{ 1 \}$.  We apply
Lemma \ref{2.8} to obtain $|\cd G| \leq |\pi (G:L)^{\leq 3}|$.
Using Lemma \ref{2.7}, we have
$$
|\pi(G:L)^{\leq 3}| \leq \textstyle {{ 16+0-1 } \choose {0}} + {{ 16+1-1 } \choose {1}} + {{16+2-1} \choose {2}} +
{{16+3-1} \choose {3}}.$$
%
This yields $|\cd G| \le 1 + 16 + 136 + 816  =  969$, and so, we may assume $L$ is not abelian.

Let $M \nor L$ be maximal with $L/M$ nonabelian.  We view that $G$ satisfies the $(2,2,\emptyset)$-hypothesis since this is the same as the two-prime hypothesis.  Thus, Theorem (5.1) of \cite{partone} applies and we obtain:
\begin{enumerate}
\item $K$ satisfies the $(1,3,\pi(G:K))$-hypothesis,
\item $L$ satisfies the $(0,5,\pi(G:L))$-hypothesis, and
\item $M$ satisfies the $(-1,9,\pi(G:M))$-hypothesis.
\end{enumerate}

Note that any nine integers $x_1, \ldots, x_9$ satisfy the condition $\omega(x_1,\ldots,x_9) > -1$, so the
$(-1,9,\pi(G:M))$-hypothesis on $M$ implies that $|\subdash {\cd M}{\pi(G:M)}| \leq 8$.

Assume $|\pi(L:M)| \leq 3$.  Then
\[
|\pi(G:M)| \leq |\pi(G:L)| + |\pi(L:M)| \leq 16+3 = 19.
\]
By Lemma \ref{9.1}, we have $\subdash {\cd G}{\pi(G:M)} = \subdash {\cd M}{\pi(G:M)}$. In particular, we determine that $|\subdash {\cd G}{\pi(G:M)}| \leq 8$.  Since $G$ satisfies the two-prime hypothesis, Lemma \ref{2.8} applies.  We obtain
$$
|\cd G| \leq 8 \cdot |\pi(G:M)^{\leq 2}| + |\pi(G:M)^3|.
$$
Using Lemma \ref{2.7}, we have
$$
|\cd G| \leq  \textstyle 8 \cdot \left[
 { {19+0-1} \choose {0}} +
 { {19+1-1} \choose {1} } +
 { {19+2-1} \choose {2} } \right] +
 { {19+3-1} \choose {3} }.
$$
It follows that $|\cd G| \le 8 \cdot \left[
1 + 19 + 190 \right] + 1330
 =  3010.$

We may suppose that $|\pi(L:M)| \geq 4$.  In particular, $L/M$ is not a $p$-group for any prime $p$.  By the choice of $L$, we have from Lemma (12.3) of \cite{text} that $L/M$ must be a Frobenius group.  Let $C/M$ be the Frobenius kernel of $L/M$.  Set $h = |L:C|$ and write $|C:M| = s^e$. Note that since $|\pi(L:M)| \geq 4$, we deduce that $\omega(h) \geq 3$.

Consider $\lambda \in \lin C$ and $\psi \in \irr {L \mid \lambda}$.  Since $L/C$ is cyclic, Corollary (11.22) from \cite{text} implies that $\psi(1)$ is the orbit-size of $\lambda$ under the action of $L/C$ on $\lin C$. Thus, $\cd {L/C'}$ is equal to the set of orbit-sizes of this action. Using Lemma (2.4) of \cite{partone}, $\cd {L/C'}$ is lcm-closed.  A set of integers $A$ is called lcm-closed if whenever $a, b \in A$, then ${\rm lcm} (a,b) \in A$.

Because $L$ is subnormal in $G$, we use Lemma \ref{9.1} to deduce that $\subdash {\cd L}{\pi(G:L)} = \subdash {\cd G}{\pi(G:L)}$.  Let $X = \subdash {\cd {L/C'}}{\pi(G:L)}$. Since $\cd {L/C'}$ is lcm-closed, $X$ is also lcm-closed.  We obtain $X \sbs \subdash {\cd G}{\pi(G:L)}$, and since $G$ satisfies the two-prime hypothesis, Lemma (3.4) of \cite{partone} implies that $X$ satisfies the two-prime hypothesis.  In light of Lemma (2.9) of \cite{partone}, we conclude that $|X| \leq 1 + 2^4 = 17$.

Fix $\lambda \in \lin C$ and $\chi \in \irr {G \mid \lambda}$.  Let $\psi \in \irr {L \mid \lambda}$ be a constituent of $\chi_L$.  By Lemma \ref{9.1}, it follows that $\chi(1)/\psi(1)$ is a $\pi(G:L)$-number.  As $\psi$ is a constituent of $\lambda^L$, we obtain $\psi \in \irr {L/C'}$. Hence, $\psi(1) = sx$, where $s$ is a $\pi(G:L)$-number and
$x \in X$.  Thus, $|\subdash {\cdover G1C}{\pi(G:L)}| = |X| \leq 17$.  From Lemmas \ref{2.7} and \ref{2.8},
$$
|\cdover G1C|  \leq 17 \cdot |\pi(G:L)^{\leq 2}| + |\pi(G:L)^3| \le 17 \cdot (1 + 17 + 153) + 969 = 3,\!876.
$$

Let $p$ be a prime. We use the notation $\cdsub pG$ for the set of character degrees of $G$ that are divisible by $p$, and $\cdsup pG$ will denote the set of character degrees of $G$ that are not divisible by $p$.

As $L$ satisfies the $(0,5,\pi(G:L))$-hypothesis,
$|\subdash {\cdsub pL}{\pi(G:L)}| \leq 4$.  Considering $L \snor G$, we conclude from Lemma \ref{9.1} that
\[
\subdash {\cdover G {\cdsub p L} L}{\pi(G:L)} = \subdash {\cdsub p L}{\pi(G:L)}.
\]
Also, because $G$ satisfies the two-prime hypothesis, so does $\cdover G{\cdsub p L}L$.  Hence, for all primes $p$, we have by Lemmas \ref{2.8} and \ref{2.7}
\[
| \cdover G{\cdsub pL}L | \leq 4 \cdot |\pi(G:L)^{\leq 2}| +
|\pi(G:L)^3| \leq 4 \cdot 153 + 816 = 1428.
\]

Fix $\chi \in \irr G$ and assume $\chi_C$ has a constituent $\theta \in \irr C$ with $\theta(1)$ divisible by $p$.  Pick $\phi \in \irr {L \mid \theta}$ to be a constituent of $\chi_L$.  Since $p$ divides $\theta(1)$ and $C \nor L$, it must be that $p$ divides $\phi(1)$. Thus, $\cdover G{\cdsub pC}C \sbs \cdover G{\cdsub pL}L$.  Therefore,
\[
|\cdover G{\cdsub pC}C| \leq |\cdover G{\cdsub pL}L| \leq 1428.
\]

Consider $x \in \cdsup sC$, where $s$ is the prime dividing the order of $C/M$, the Frobenius kernel of $L/M$.  Theorem 12.4 of \cite{text} applied to the group $L/M$ implies that $hx \in \cd L$.  Recall that $\omega (h) \geq 3$.   Hence, we have $\cdover G{ \{ h,hx \} }L = \{ a\}$ for some number $a$ by the two-prime hypothesis in $G$.  Since $L \snor G$, we apply Lemma \ref {9.1} to see that
\[
\subdash h{\pi(G:L)} = \subdash a{\pi(G:L)} = \subdash
{(hx)}{\pi(G:L)}.
\]
In particular $x$ is a $\pi(G:L)$-number.  We have proved that for every degree $x \in \cd C$, either $x$ is divisible by $s$ or $x$ is a $\pi (G:L)$-number.  In particular, either $x = 1$ or some prime in $\pi (G:L) \cup \{ s \}$ divides $x$.  Recall that $|\pi(G:L)| \leq 16$.  Write $\sigma = \pi(G:L) \cup \{ s \}$ and note that $|\sigma| \leq 17$.

Now, every degree $a \in \cd G$ lies over some $x \in \cd C$.  We have shown that any such $x$ is equal to $1$ or is divisible by some element of $\sigma$.  Thus, $|\cd G| \leq | \cdover G1C | + \sum_{p \in \sigma} |\cdover G{\cdsub pC}C|$.
This implies that $|\cd G| \leq 3264 + |\sigma| \cdot 1428 \leq 3264 + 17 \cdot 1428 = 27,\!540.$  This is the largest bound we have yet obtained, and we have exhausted all cases.  Thus, $|\cd G| \leq 27,\!540$.
\end{proof}

\section{The Small Frobenius Case, part 2}

In the next two sections, we will consider the following situation.

\medskip
{\bf Hypothesis (``$K$ Nonabelian'')}.  Assume the $f$ Small
Hypothesis from Section \ref{Chapter 9}. Assume $K$ is nonabelian, and write $\pi = \pi(G:K)$.  Assume $K' \sbs \Oh p K$ for all primes $p$.  Let $L \nor G$ with $L \sbs K$ and suppose $L$ is chosen maximal with $K/L$ nonabelian.  Assume that for all such choices of $L$ we have $|\pi(K:L) - \pi| \geq 9$.  Write $C/L = (K/L)'$ and $K/C = K_0/C \times H/C$, where $K_0/C = \oh {\pi'} {K/C}$ and $H/C = \oh {\pi} {K/C}$. Let $g = |K_0:C|$ and $h = |H:C|$.  Assume $L$ is chosen so that $gh$ is minimized. Write $|C:L| = r^e$.  Suppose there exists $b \in \cd G$ divisible by $g$.
\medskip

Under the $K$ Nonabelian Hypothesis, $|\pi(g)| \geq 8$, and so, by the two-prime hypothesis in $G$, we see that $b$ is the
unique character degree of $G$ divisible by $g$.  In view of Lemma \ref{7.3}, Lemma \ref{7.2} applies to the group $K/L$.  Now, Lemma \ref {7.2} (b) implies that $gh \in \cd K$.  In particular, we may take $b \in \cdover G{gh}K$.

We now consider $K$ Nonabelian hypothesis, and we show that we have a quotient that is a Frobenius group.

\begin{lemma} \label{10.1}
Assume the $K$ Nonabelian Hypothesis.  Then $K_0/L$ is a Frobenius group with kernel $C/L$.
\end{lemma}

\begin{proof}
Assume otherwise. Let $1 < n < gh$ be an orbit size of the action of $K/C$ on $C/L$.  Lemma \ref{7.3} implies that $C/L$ is the unique minimal characteristic subgroup of $K/L$.  By Lemma \ref{7.2} (c), $\pi(n) = \pi(gh)$, and so $\omega(gh, n) \geq |\pi(gh)| \geq 8$.

Since $G$ satisfies the two-prime hypothesis, Lemma (3.6) of \cite{partone} implies that $K$ satisfies the $(2,\pi)$-hypothesis.  Hence, we have $\subdash {(gh)}\pi = \subdash n\pi$.  Since $g$ is a $\pi'$-number and $h$ is a $\pi$-number, we obtain $\subdash n\pi = g$.

Suppose $1 \not= x \in C/L$.  We want to show that $\cent {K_0/L} x = 1$.  Now, $|K_0/L : \cent {K_0/L} x|$ is the $\pi'$-part of the size of the orbit of $x$ under the action of all of $K/C$.  If the $K/C$-orbit of $x$ is nontrivial, then from the previous paragraph we know that $\cent {K_0/L} x = 1$.

Assume $x$ lies in a trivial $K/C$-orbit.  Then $1 \not= x \sbs \zent {K/L}$.  As $C/L$ is the unique minimal characteristic subgroup of $K/L$, we obtain $C/L \sbs \zent {K/L}$.  However, Lemma \ref{7.2} (b) implies that $K/C$ acts on $C/L$ with a regular orbit, and in particular, this action is nontrivial.  This is a contradiction, and completes the proof.
\end{proof}

We continue to study the $K$ Nonabelian Hypothesis.

\begin{lemma} \label{10.2}
Assume the $K$ Nonabelian Hypothesis.  Let $M \nor G$ with $M
\sbs L$.  Assume that $L/M$ is an abelian $p$-group for some prime $p \neq r$.  Then $L/M \sbs \zent {C/M}$.
\end{lemma}

\begin{proof}
Assume that $C/L$ acts nontrivially on $L/M$. Since this
is a coprime action on an abelian group, apply Fitting's Theorem. Write $L/M = A/M \times B/M$, where $A/M = \cent {L/M} {C/L}$ and $B/M = [L/M, C/L]$.  Because the action is nontrivial, $A < L$.  Let $L/M_0$ be a $G$-chief factor of $L/A$.  The action of $C/L$ on $L/M_0$ cannot be trivial, as otherwise $M_0 \sps B$, which is a contradiction.

By Lemma \ref{10.1}, we have that $K_0/L$ is a Frobenius group with kernel $C/L$.  In particular, $K_0/C$ is cyclic.  Let $F/L$ be a Frobenius complement in $K_0/L$. Now consider the action of $K_0/L$ on the $p$-group $\irr {L/M_0}$.  We have by Theorem (15.16) of \cite{text} that $L/M_0$ has a basis which is permuted semi-regularly by $F/L$. Let $u \in \pi(g)$, and write $a = g/u$.  By Lemma \ref{7.6}, there is an $F/L$-orbit of size $a$ in $\irr {L/M_0}$.  Let $\lambda \in \irr{L/M_0}$ be an element of this orbit.  Applying Lemma (2.5) of \cite{partone}, $a$ is the $r'$-part of the size of the $K_0/L$-orbit of $\lambda$.

Because $|C/L|$ is coprime to $|L/M_0|$ and $K_0/C$ is cyclic, it follows by Corollaries (6.28) and (11.22) of \cite{text}
that $\lambda$ extends to $I_{K_0}(\lambda)$.  Now, $a = \subdash {|K_0:I_{K_0}(\lambda)|}r$, so by the Clifford correspondence, $a \in \subdash {\cd {K_0}}r$.  Let $c \in \cd {K_0}$ with $\subdash cr = a$.

Since $G$ satisfies the two-prime hypothesis and $\pi(G:K_0) = \pi$, we apply Lemma (3.6) of \cite{partone} to see that $K_0$ satisfies the $(2,\pi)$-hypothesis.  Since $K_0/L$ is a Frobenius group, we have $g \in \cd {K_0}$.  Now, $\omega(c,g)
\geq \omega(g/u) \geq 7$, where the last inequality follows since $\omega(g) \geq 8$.  Thus, $\subdash c\pi = \subdash g\pi$.  Because $g$ is a $\pi'$-number, $\subdash g\pi = g$.  Because $K_0/L$ is a Frobenius group, $r$ does not divide $g$. Finally, we obtain $g = a = g/u$, which is a contradiction.
\end{proof}

Continuing to study the $K$ Nonabelian hypothesis,  we need to consider the case where $L$ is abelian.  The proof of a bound in the case where $L$ is nilpotent is only slightly more difficult, and we present the more general result.

\begin{theorem} \label{10.3}
Assume the $K$ Nonabelian Hypothesis.  Assume $L$ is
nilpotent. Then $|\cd G| \leq 221,\!205$.
\end{theorem}

\begin{proof}
Assume $L$ is nilpotent and write $L = R \times S$, where
$R = \oh rL$ and $S = \oh {r'}L$.  Let $U \in \syl rC$, and note that $U$ is a complement for $S$ in $C$.  Consider the action of $U$ on $S$. Since $R$ acts trivially on $S$, this action is the same as the action of $U/R$ on $L/R$.  Write $W/R = \Phi(L/R)$.  Note that $L/W$ is a direct product of elementary abelian $p$-groups for primes $p \ne r$.  By Lemma \ref{10.2}, $U/R$ acts trivially on $\Phi(L/R)$. Since this is a coprime action, it must be that $U/R$ acts trivially on all of $L/R$.  Thus, $U$ centralizes $S$ and $C = U \times S$.  In particular, $C$ is nilpotent.

Now, $K_0$ satisfies the $(2,\pi)$-hypothesis by Lemma (3.6) of \cite{partone}.  Using Lemma \ref{10.1}, se see that the hypotheses of Theorem (6.3) of \cite{partone} apply to the group $K_0$, and we obtain $|\subdash {\cd {K_0}}\pi| \leq (1 + 2^4)^3 = 4,\!913$.

Finally, since $|G:K_0|$ is a $\pi$-number, it follows that
$\subdash {\cd G}\pi = \subdash {\cd {K_0}}\pi$.  Finally, by Lemmas \ref{2.7} and \ref{2.8}, we conclude that
$$
|\cd G| \leq 4,\!913 \cdot | \pi^{\leq 2} | + |\pi^3| \le \textstyle 4,\!913 \cdot \left[ 1 + 8 + 36  \right]  + 120 =  221,\!205.
$$
\end{proof}

We continue to study the $K$ Nonabelian hypothesis.

\begin{lemma} \label {10.4}
Assume the $K$ Nonabelian Hypothesis.  Assume $f$ is prime and
that $N/K_0$ is a nonabelian $q$-group.  Then $|\cd G| \leq 268$.
\end{lemma}

\begin{proof}
Observe that $N$ satisfies the $(2,\{ f \})$-hypothesis by Lemma (3.6) of \cite{partone}.  Since $\Oh qN \sbs K_0$, it follows that $N/\Oh qN$ is nonabelian.  Then $\Oh qN$ satisfies the $(1,\{ f, q \})$-hypothesis by Lemma (3.2) of \cite{partone}.  For the rest of this proof, take $\pi = \{ f, q \}$.  In light of Theorem D from \cite{partone}, $|\subdash {\cd {\Oh qN}}\pi| \leq \frac 32 \cdot 2^2 + \frac {19}2 \cdot 2 + 18 = 43$. Since $|G:\Oh qN|$
is a $\pi$-number, we have $\subdash {\cd G}\pi = \subdash {\cd {\Oh qN}}\pi$, and hence $|\subdash {\cd G}\pi| \leq 43$. By Lemmas \ref{2.7} and \ref{2.8}, we obtain
\[
|\cd G| \leq 43 \cdot |\pi^{\leq 2}| + |\pi^3| = 43 \cdot [ 1 + 2 + 3] + 10 = 268.
\]
\end{proof}

The next lemma contains much of the work that we have to do in the $K$ Nonabelian hypothesis.

\begin{lemma} \label{10.5}
Assume the $K$ Nonabelian Hypothesis.  Let $\scr Y$ be the set of characters in $\irr L$ that have prime degree and that extend to $C$.  Then $|\cd {G \mid \scr Y}| \leq 5,\!697$.
\end{lemma}

We prove this lemma in a series of steps.

\setcounter{stepcount}{1}

\begin{proof}
Let $\psi \in \scr Y$, and let $\hat\psi \in \irr C$ be an extension $\psi$. Write $\psi(1) = p$.  Recall that $b \in \cd G$ is the unique degree of $G$ divisible by $g$.  Also, note that by Lemma \ref{7.3}, Lemma \ref{7.2} applies to the group $K/L$.  In particular, we have $gh \in \cd K$.


\step $pg \in \cd {K_0}$.

\begin{proof}
By Lemma \ref{10.1}, $K_0/L$ is a Frobenius group. It follows from Theorem 12.4 of \cite{text} that either $\V {\hat\psi} \sbs L$ or $pg \in \cd {K_0}$.  Since $\hat\psi_L = \psi$ is irreducible, we cannot have $\V {\hat\psi} \sbs L$.  Thus, $pg \in \cd {K_0}$, as desired.
\end{proof}


\step $p \in \pi$.

\begin{proof}
Let $d \in \cd G$ lie over $pg \in \cd {K_0}$.  Then $g$ divides $d$, and so $d=b$.  Now, $b$ also lies over $gh \in \cd K$. Since both $h$ and $|G:K|$ are $\pi$-numbers, $b/g$ is a
$\pi$-number.  We conclude that $p$ divides $b/g$, and so $p \in \pi$.
\end{proof}


Write $T = I_{K_0}(\psi)$ and $t = |K_0:T|$.

\step $\psi$ extends to $T$, and $\omega(t) \leq 1$ or $t=g$.

\begin{proof}
Since $K_0/L$ is a Frobenius group, $K_0/C$ is cyclic, and in particular, $T/C$ is cyclic.  Now, $\psi$ extends to $C$, and since $|T/C|$ is coprime to $|C|$, we have by Corollary (11.22) of \cite{text} that $\psi$ extends to all Sylow subgroups of $T$ for primes not dividing $|C|$. Hence $\psi$ extends to $T$.

Assume $\omega(K_0:T) \geq 2$.  We want to show that in this case $t=g$.  By the Clifford Correspondence, $pt \in \cd {K_0}$.  Now, $t$ divides $|K_0:C| = g$, and so $pt$ divides $pg$.  In particular, $\omega (pg,pt) = \omega (pt) \geq 3$. As $|G:K_0|$ is a $\pi$-number, $K_0$ satisfies the $(2,\pi)$-hypothesis by Lemma (3.6) of \cite{partone}.  It follows that $\subdash {(pg)}\pi = \subdash {(pt)}\pi$.  By Step 2, $p \in \pi$.  Also, $g$ is a $\pi'$-number, and since $t$ divides $g$, we have $t$ is also a $\pi'$-number.  We conclude that $g=t$, as desired.
\end{proof}


\step All degrees $a \in \cd {G \mid \psi}$ have the form $a=xt$ or $a=xg$, where $x$ is a $\pi$-number divisible by $p$.

\begin{proof}
By Step 3, $\psi$ extends to $T$.  If $T > C$, then $T/L$ is a Frobenius group and $\cd {T/L} = \{ 1, g/t \}$. If $T=C$, then $g/t = 1$, and since $T/L = C/L$ is abelian, we again have $\cd {T/L} = \{ 1, g/t \}$. Thus, by Gallagher's theorem, $\cd {T \mid \psi} = \{ p, pg/t \}$. Applying the Clifford Correspondence, we obtain $\cd {K_0 \mid \psi} = \{ pt, pg \}$.  Since $|G:K_0|$ is a $\pi$-number, it follows that all degrees in $\cd {G \mid \psi}$ have the form $ypt$ or $ypg$, where $y$ is a $\pi$-number.  Since $p \in \pi$, the step is proved.
\end{proof}

\medskip
For $\psi \in \scr Y$, write $t(\psi) = |K_0:I_{K_0}(\psi)|$. Define $\scr A = \{ \psi \in \scr Y \mid t(\psi) = 1 \}$, $\scr B = \{ \psi \in \scr Y \mid t(\psi) \mbox{ is prime}\}$, and $\scr C = \{ \psi \in \scr Y \mid t(\psi) = g\}$.  We have by Step 3 that $\scr Y = \scr A \cup \scr B \cup \scr C$.


\step $|\cd {G \mid \scr A}| \leq 166$.

\begin{proof}
Let $\psi \in \scr A$, and note that $T = K_0$.  Now, $\psi$ extends to $T$ by Step 3.  Using Gallagher's Theorem, $\cd {K_0 \mid \psi}$ consists of degrees of the form $pa$, where $a \in \cd {K_0/L}$.  Now since $K_0/L$ is a Frobenius group, we have $\cd {K_0/L} = \{ 1, g\}$.  Thus, $\cd {K_0 \mid \psi} = \{ p, pg \}$.

Note that $b$ is the unique character degree of $G$ which lies over $pg$.  Let $d \in \cd G$ lie over $p \in \cd {K_0}$.  Then $d/p$ divides $|G:K_0|$, which is a $\pi$-number.  Since also $p \in \pi$, we see that $d$ is a $\pi$-number.

Let $X = \cd {G \mid \scr A} - \{ b \}$.  Observe that $X$ is a set of $\pi$-numbers that satisfies the two-prime hypothesis since $X \sbs \cd G$.  Since $|\pi| \leq 8$, we apply Lemmas \ref{2.7} and \ref{2.8} to compute
\[ \textstyle
|X| \leq |\pi^{\le 3}| = |\pi^0| + |\pi^1| + |\pi^2| + |\pi^3| \leq 1 + 8 + 36 + 120 = 165.
\]
We conclude that $|\cd {G \mid \scr A}| \leq |X| + 1 \leq 166$, as desired.
\end{proof}


\step $|\cd {G \mid \scr C}| \leq 1$.

\begin{proof}
Let $\psi \in \scr C$, and note that $I_{K_0}(\psi) = C$. By the Clifford Correspondence, $|K_0:C| = g$ divides all degrees in $\cd {K_0 \mid \psi}$.  Hence, $g$ divides all elements of $\cd {G \mid \scr C}$.  Since $b$ is the unique degree of $G$ divisible by $g$, we deduce that $|\cd {G \mid \scr C}| \leq 1$, as desired.
\end{proof}

\medskip
Let $\psi \in \scr B$, and write $S = I_K (\psi)$. Recall that $T = I_{K_0} (\psi)$ and $t = |K_0:T|$.  We now consider two main cases.  In Case 1, we assume $H \not\sbs S$ or $\psi$ does not extend to $S$.  In Case 2, we assume both $H \sbs S$ and $\psi$ extends to $S$.  Write $\scr B_1 \sbs \scr B$ for the set of characters for which Case 1 holds.  Write $\scr B_2 \sbs \scr B$ for the set of characters for which Case 2 holds.


\step Assume $\psi \in \scr B_1$. Then there exists $xt \in \cd G$, where $x$ is a $\pi$-number with $\omega(x) \geq 3$.

\begin{proof}
First note that Step 3 implies that $\psi$ extends to $\chi \in \irr T$. Since $|S:T|$ divides $|K:K_0|$,  the index $|S:T|$ is a $\pi$-number.  Let $pu \in \cd {S|\chi}$, where $u$ is a divisor of $|S:T|$.  Applying the Clifford Correspondence, $pu|K:S| \in \cd K$.  Note that $|K:S| = tv$, where $v$ is a $\pi$-number. Thus, $ptuv \in \cd K$.

Suppose $H \not\sbs S$.  Then $|K:S| = tv$ has nontrivial
$\pi$-part, and in particular, $v > 1$.  Suppose that $\psi$ does not extend to $S$.  Then we cannot have $p \in \irr {S|\psi}$, and so, in this case $u > 1$.  We obtain $ptw \in \cd K$, where $w = uv$ is a $\pi$-number and $w > 1$.

Because $G/K$ is a nonabelian $\pi$-group, we may choose $ptwz \in \cd G$ lying over $ptw \in \cd K$ such that $z > 1$.  If we write $x = pwz$, then $x$ is a $\pi$-number with $\omega(x) \geq 3$ and $tx \in \cd G$, as desired.
\end{proof}


\step The set $\{ t(\psi) \mid \psi \in \scr B_1 \}$ has at most 120 elements.

\begin{proof}
For all $\psi \in \scr B_1$, use Step 7 to obtain $t(\psi)x(\psi) \in \cd G$, where $x(\psi)$ is a $\pi$-number depending on $\psi$ with $\omega(x(\psi)) \geq 3$.  The set of these $t(\psi)x(\psi)$ is contained in $\cd G$, and hence, satisfies the two-prime hypothesis.

Assume that  $\psi, \psi' \in \scr B_1$ and that there is some integer $u \in \pi^3$ such that $u$ divides both $x(\psi)$ and $x(\psi')$.  Then $\omega(t(\psi)x(\psi), t(\psi')x(\psi')) \geq \omega(u) = 3$.  By the two-prime hypothesis, we deduce that $t(\psi)x(\psi) = t(\psi')x(\psi')$, and it must be that $t(\psi) = t(\psi')$, since these are $\pi$-parts of equal numbers.

Because each $x(\psi)$ has $\omega(x(\psi)) \geq 3$, each
$x(\psi)$ is divisible by some number $u \in \pi^3$.  In light of the previous paragraph, it follows that there are at most $|\pi^3|$ distinct numbers $t(\psi)$.  By Lemma \ref{2.7}, we may use the fact that $|\pi| \le 8$ to see that $|\pi^3| \leq {{8+3-1}\choose 3} = 120$, and this proves the step.
\end{proof}


\step $|\cd {G \mid \scr B_1}| \leq 5,\!521$.

\begin{proof}
By Step 4, all degrees $a \in \cd {G \mid \scr B_1}$ have the form $a = t(\psi)x$ or $a = gx$, where $\psi \in \scr B_1$ and $x$ is a $\pi$-number.  If $a = gx$, then $g$ divides $a$ and it follows that $a = b$.  Write $X = \cd {G \mid \scr B_1} \setminus \{ b \}$.  Then $\subdash X\pi$ is the set of $t(\psi)$ for $\psi \in \scr B_1$.  By Step 8, we have  $|\subdash X\pi| \leq 120$.  Hence, by Lemmas \ref{2.7} and \ref{2.8}, we compute
\[ \textstyle
|X| \leq 120 \cdot |\pi^{\leq 2}| + |\pi^{\leq 3}| \leq
120 \cdot \left[ 1 + 8 + 36 \right] + 120 = 5,\!520.
\]
We obtain the conclusion that $|\cd {G \mid \scr B_1}| = |X| + 1 \leq 5,\!521$.
\end{proof}

\medskip
Recall that we have written $S = I_K(\psi)$, $T = I_{K_0} (\psi)$, and $t = |K_0:T|$.  We now consider characters $\psi \in \scr B_2$, that is, $H \sbs S$ and $\psi$ extends to $S$.  Note that in this case $|K:S| = |K_0:T| = t$.


\step Assume $\psi \in \scr B_2$.  Then $\cd {K \mid \psi}$ is
exactly the set of numbers $a$ of the form $a = pt$ or $a = pgh_0$, where $h_0 \in \cd {H/L}$.  In particular, $pt, pgh \in \cd {K \mid \psi}$.

\begin{proof}
We first claim that $\cd {S/L}$ is the set of degrees of the form $(g/t)h_0$, where $h_0 \in \cd {H/L}$. Note that since $\psi \in \scr B$, we have $|K_0:T|$ is prime, and in particular $T > C$. Thus, $T/L$ is a Frobenius group with kernel $C/L$, and hence, $\cd {T/L} = \{ 1, g/t\}$.

We now work to show that all nonlinear characters of $S/L$ have degree divisible by $g/t$.  Let $\chi \in \irr {S/L}$ and assume $\chi_{T}$ has constituents of degree $g/t$.  Then clearly $g/t$ divides $\chi(1)$.  Assume $\chi_{T}$ has linear constituents.  We have $C/L = (T/L)'$, and so, $\chi_C$ is a multiple of the principal character.  In particular, $\chi \in \irr {S/C}$, and since $S/C$ is abelian, $\chi$ is linear, as desired.

Suppose $\chi \in \irr {S/L}$ is nonlinear.  Then $\chi(1) = (g/t)x$ for some integer $x$ which divides $|S:T|$. Since $|S:T|$ is coprime to $|S:H|$, it follows that $\chi_H$ has constituents of degree $x$ and the claim is proved.

We have that $\psi$ extends to $S$.  Hence, by Gallagher's Theorem, the elements of $\cd {S \mid \psi}$ have the form $p$ or $p(g/t)h_0$ for $h_0 \in \cd {H/L}$. By the Clifford correspondence, $\cd {K \mid \psi} = \{ pt \} \cup \{ pgh_0 \mid h_0 \in \cd {H/L}\}$, as desired.  Applying Lemma \ref{7.2} (b), $K/C$ acts on $C/L$ with a regular orbit.  In particular, $H/C$ acts on $C/L$ with a regular orbit.  It follows by Lemma \ref{7.1} that $h \in \cd {H/L}$, and thus, $pgh \in \cd {K \mid \psi}$.  This completes the proof of Step 10.
\end{proof}

\medskip
From Step 10, we know, in particular, that $pt \in \cd K$.  Since $G/K$ is a Frobenius group, it follows from Theorem 12.4 of \cite{text} that we may choose $c \in \cd G$ lying over $pt \in \cd K$ such that $ptf$ or $ptq$ divides $c$.  We now work to prove that we can reduce to the case where $f$ is not prime, $ptf$ divides $c$, and $ptq$ does not divide $c$.


\step Assume $\psi \in \scr B_2$ and let $c \in \cdover G{pt}K$. Assume $ptf$ divides $c$.  Then $f$ is not prime.

\begin{proof}
Assume otherwise.  By Step 10, $pgh \in \cd K$, and so $b$ divides $pgh|G:K|$.  Write $b = pghf^nq^m$ where $0 \leq n \leq 1$ and $m \geq 0$.  Using Theorem 12.4 of \cite{text}, we may choose $d \in \cdover G {pgh} K$ with $d/pgh$ divisible by $q$ or $f$.  As this $d$ is divisible by $g$, we have that $d = b$ and $b/pgh$ is divisible by $p$ or $f$.  Hence, $n > 0$ or $m > 0$.

Since $ptf$ divides $c$ and $c$ divides $pt|G:K|$, write $c = ptfq^e$ where $e \geq 0$.  Now, $\subdash c\pi = t$ and $\subdash b\pi = g$, but $t$ is prime since $\psi \in \scr B_2$.  Hence, $b \ne c$, and so, $\omega(b,c) \leq 2$ by the two-prime hypothesis in $G$.  However, clearly $pt$ divides $(b,c)$ and so $(b/pt, c/pt) = 1$.  It follows that $n = 0$ and $f$ does not divide $h$.   Also, we have $m \geq 1$ and $e =
0$.  Because $h$ is a $\pi$-number and $\pi = \{ f,q \}$, it must be that $h$ is a power of $q$.

Now, $N/K_0$ is a $q$-group.  If $N/K_0$ is nonabelian, then by Lemma \ref{10.4}, $|\cd G| \leq 268$, and the theorem is true in this case.  Thus, we may assume $N/K_0$ is abelian.

Note that $K/H \cong K_0/L$ is cyclic, and so $(G/H)/\cent {G/H}{K/H}$ is abelian.  Since $G/K$ is a Frobenius group and $\cent {G/H}{K/H} \nor G/H$, this implies that $N/H \sbs \cent {G/H}{K/H}$.  In particular, since $N/K$ is a $q$-group and $q$ does not divide $|K:H|$, we must have $N/H$ is abelian and so $N' \sbs H$.  Therefore, $N' \sbs K_0 \cap H = C$, and so $N/C$ is abelian.  Write $N/C = K_0/C \times Q/C$, where $Q/C$ is a $q$-group.

Assume $Q/L$ is a $q$-group.  Since $C/L$ is a $G$-chief factor, we must have $C/L \sbs \zent {Q/L}$.  Also, $Q/C$ is abelian, so $(Q/L)' \sbs C/L$.  Again $C/L$ is a $G$-chief factor, so either $Q/L$ is abelian or $C/L = (Q/L)'$.  First assume $Q/L$ is abelian, and let $\lambda \in \irr {C/L}$ be nonprincipal.  Now, $\lambda$ extends to $Q$, and since $K_0/C$ acts Frobeniusly on $C/L$, we also have $I_N(\lambda) = Q$.  Applying the Clifford correspondence, $g \in \cd N$.  Then $b$ divides $g|G:N| = gf$ which cannot occur since we have proved $pqgh$ divides $b$.  Hence, we may assume $(Q/L)' = C/L$,
and since $C/L$ is abelian and central in $Q/L$, we consider the action of $K_0/C$ on $Q/L$.  Under this coprime action, $K_0/C$ centralizes $Q/C = (Q/L)/(Q/L)'$ and acts nontrivally on $(Q/L)'$.  This is a contradiction, and therefore $Q/L$ is not a $q$-group.

Now, $C \sbs H \sbs Q$ and $Q/H$ is $G$-isomorphic to $N/K$, so $Q/H$ is a $G$-chief factor.  Let $U/L = \cent {Q/L} {C/L}$.  Observe that $H/C$ acts faithfully on $C/L$ by Lemma \ref{7.2} (b), so $U \cap H = C$.  Since $U \nor G$, we have either $UH=Q$ or $U=C$.  Let $1 \not= \lambda \in \irr {C/L}$ and write $V = I_G(\lambda)$.

First assume $UH = Q$.  Since $q$ does not divide $|C:L|$, it must be that $\lambda$ extends to $U \sbs V$.  Let $x \in \cd
{V  \mid \lambda}$ with $x$ dividing $|V:U|$.  Now, $g$ divides $|G:V|$ since $K_0/C$ acts Frobeniusly on $C/L$.  By the Clifford correspondence, $x|G:V| \in \cd G$, and since this degree is divisible by $g$, we have $b = x|G:V|$.  Hence $b$ divides $|G:U|$. However, $|G:U| = fgh$ and we have proved $pqgh$ divides $b$.  This is a contradiction.

Hence, we may assume $U=C$.  Then $Q/C$ is an abelian $q$-group acting faithfully on $C/L$.  Since $K_0/C$ is cyclic and
$(|Q:C|,|C:L|) = 1$, we have that $\lambda$ extends to $V \cap N$ by Corollaries (11.22) and (6.28) of \cite{text}.  Thus, the characters of $N$ lying above $\lambda$ have degree $gq^a$ where $q^a$ is exactly the size of the $Q/C$-orbit of $\lambda$. Now, $|G:N|$ is a $q'$-number, so $q^a$ must also be the $q$-part of $b$.  Thus, $Q/C$ acts half-transitively on $C/L$.  By Lemma 2.6 from \cite{riedl}, it follows that $Q/C$ is cyclic. Since the action is faithful and $C/L$ is a $G$-chief factor, it must be that the action of $Q/C$ on $C/L$ is Frobenius.

Since $N/C$ acts Frobeniusly on $C/L$, we obtain $V \cap N = C$.  Thus, $V/C \cong VN/N$ is cyclic since $G/N$ is cyclic.  It follows by Corollary (11.22) from \cite{text}
that $\lambda$ extends to $V$.  Using the Clifford Correspondence, $|G:V| \in \cd G$, and as $g$ divides $|G:V|$, we have $b = |G:V|$.

Because $N/K$ is both cyclic and elementary abelian, $|N:K|=q$ and hence $|G:K| = fq$. Since we have both $\omega(b/gh) \geq 2$ and $b/gh$ divides $|G:K|$, we see that $b = fqgh$.  Thus, $|G:V| = b = |G:C|$ and so $V=C$.  By Lemma \ref{7.1}, we conclude that $G/C$ acts Frobeniusly on $C/L$.

We have $\psi \in \irr L$ with $\psi(1) = p$ and $\psi$ extends to $\hat \psi \in \irr C$.  Clearly, $\V {\hat\psi} \not\sbs L$, so $p|G:C| \in \cd G$ by the Theorem 12.4 of \cite{text} applied to $G/L$.  Because $g$ divides $p|G:C|$, we see that $b = p|G:C|$.  Also $b = |G:C|$, which is a contradiction.  This completes the proof of Step 11.
\end{proof}


\step Assume $\psi \in \scr B_2$.  Let $c \in \cdover {G}{pt}K$ with $c > pt$.  Then $c$ is not divisible by $ptq$.

\begin{proof}
First, observe that $pt \in \cd K$ by Step 10.  Also, since $G/K$ is nonabelian, we may choose $c \in \cd G$ lying above $pt$ with $c > pt$.   Assume that $ptq$ divides $c$.

Write $c = ptf_0 q^a$ where $a \geq 1$ and $f_0$ is a divisor of $f$. Recall that by Step 10, $pgh \in \cd K$. Since $G/K$ is a Frobenius group, we use Theorem 12.4 of \cite{text} to choose $d \in \cdover G{pgh}K$ such that $d/pgh$ is divisible by either $q$ or $f$.  Since $g$ divides $d$, it must be that $d=b$ and $b/pgh$ is divisible by either $q$ or $f$.   Hence, we may write $b = pghf_1 q^m$, where $m \geq 0$ and $f_1$ is a divisor of $f$. We must have either $f_1 = f$ or $m > 0$.

Recall that since $\psi \in \scr B_2$, the index $t = |K_0:T|$ is prime, where we are continuing to write $T = I_{K_0}(\psi)$.  Because $\subdash b\pi = g$ and $\subdash c\pi = t$, we cannot have $b = c$, as before.  Applying the two-prime hypothesis in $G$, we deduce that $\omega(b,c) \leq 2$.  Now, $pt$ divides $(b,c)$, and so, $f_0q^a$ and $hf_1 q^m$ are coprime integers. Thus, $m = 0$, $q$ does not divide $h$, and $f_1 = f$.  We see that $b = fpgh$ and $f_0 = 1$.  Also, $b \in \cdover G{gh}K$, so $pf$ divides $|G:K|$.  This forces $p = q$.  Since $h$ is a $\pi$-number and $q$ does not divide $h$, it follows that
$\pi(h) \sbs \pi(f)$.

We next work to prove that $K/C$ acts Frobeniusly on $C/L$. Recall that the character degrees of $K/L$ are exactly the numbers $gh_0$, where $h_0$ divides $h$ and $gh_0$ is the size of a $K/C$-orbit in $\irr {C/L}$.  Let $gh_0 \in \cd K$ where $h_0$ is some divisor of $h$.  Then $b \in \cdover G{gh_0}K$ and so $b$ divides $gh_0|G:K|$.  Since $q^2$ does not divide $b$, we see that $b$ divides $gh_0\cdot fq$.  Now, $b = fghq$, which implies $h=h_0$.  Hence, $gh$ is the only nontrivial irreducible character degree of $K/L$, and in particular, $gh$ is the only nontrivial orbit-size of the action of $K/C$ on $C/L$.  We deduce that $K/L$ is a Frobenius group with kernel $C/L$.

In particular, since $K/C$ is abelian, it follows that $K/C$ must be cyclic.  Now, $(G/C)/\cent {G/C} {K/C}$ must be abelian.  We must have $N/C \sbs \cent {G/C} {K/C}$ since $G/K$ is a Frobenius group.  Recall that $q$ does not divide $h$, and so $q$ does not divide $|K:C|$.  We conclude that $N/C$ is abelian.

Since $N/C$ is abelian, we may write $N/C = K/C \times Q/C$, where $Q/C$ is a Sylow $q$-subgroup of $N/C$.  Because $N/K$ is a nontrivial $q$-group, $Q > C$.  We want to show that if $Q/L$ is a $q$-group, then $Q/L$ is abelian.  Assume otherwise.  As $Q/C$ is abelian, it must be that $(Q/L)' = Q/C$.  Now, $K/C$ acts coprimely on $Q/L$.  Observe that $K/C$ acts trivially on $Q/C$ but faithfully on $C/L$, which is a contradiction.  Therefore, if $Q/L$ is a $q$-group, then $Q/L$ is abelian.

Let $B/C = \cent {N/C} {C/L}$.  Since $K/C$ acts Frobeniusly on $C/L$, we have $B \cap K = C$.  Also, $N/K \cong Q/C$ is a $G$-chief factor and $B \nor G$.  Thus, either $B = Q$ or $B = C$.  Let $1 \ne \lambda \in \irr {C/L}$, and write $V = I_G(\lambda)$.

Suppose $B = C$.  Then the abelian group $N/C$ acts faithfully on $C/L$, and we would like to show that the action is coprime.  By Lemma \ref{7.2} applied to the group $K/L$, we know that $|K/C|$ is coprime to $|C/L|$.  Assume that $|C/L|$ is a power of $q$.  Then $Q/L$ is a $q$-group, and we have proved that in this case $Q/L$ is abelian. Then $B \sps Q > C$, which is a contradiction.  Thus, in the case where $B=C$, the action of $N/C$ on $C/L$ is faithful and coprime.

{}From Lemma (2.3) of \cite{partone}, the action of $N/C$ on $C/L$ has a regular orbit.  In light of Lemma \ref{7.1}, $N/C$ has a regular orbit in its action on $\irr {C/L}$.  It follows that $|N:C| \in \cd N$.  However, $|N:C| = ghq^e$, so $ghq^e$ divides $b$.  Since $q^2$ does not divide $b$, we must have $e = 1$. Therefore, $|G:C| = fqgh = b$.

Observe that $\lambda$ has a uniquely determined extension to $V \cap N$ by Corollary (6.28) of \cite{text} since $|(V\cap N)/C|$ is coprime to $|C/L|$.  Furthermore, $VN/N \cong V/(V\cap N)$ is cyclic, so by Corollary (11.22) of \cite{text}, $\lambda$ extends to $V$.  The Clifford Correspondence yields $|G:V| \in \cd G$.  Since $K_0/C$ acts Frobeniusly on $C/L$, we may use Lemma \ref{7.1} to see that $g$ divides $|G:V|$. Hence $b = |G:V|$.

In this case, we have proved $|G:V| = b = |G:C|$.  Since $C \sbs V$, it follows that $V=C$.  We have proved for all nonprincipal characters $\lambda \in \irr {C/L}$ that $I_G(\lambda) = C$.  It follows by Lemma \ref{7.1} that $G/C$ acts Frobeniusly on $C/L$.

We have $\psi \in \irr L$ with $\psi(1) = p$, and $\psi$ extends to $\hat \psi \in \irr C$.  Clearly, $\V {\hat\psi} \not\sbs L$, so we may apply Theorem 12.4 of \cite{text} to $G/L$ to obtain $|G:C|p \in \cd G$.  But, $g$ divides $|G:C|p$, so $|G:C| = b = |G:C|p$, which is a contradiction.

We may now assume $B=Q$.  If $Q/L$ is a $q$-group, then $Q/L$ is abelian, and so, $\lambda$ extends to $Q$.  Otherwise, $|Q/C|$ is coprime to $|C/L|$, and so $\lambda$ extends to $Q$ by Corollary (6.28) of \cite{text}.  Now, $B \sbs V$, so the Clifford correspondence implies that $|G:V|x \in \cd
{G|\lambda}$ where $x$ divides $|V:B|$.  As $g$ divides $|G:V|$, we see that $|G:V|x = b$.  In particular, $b$ divides $|G:B| = fgh$. But $b = fghq$, which is a contradiction.  We conclude that $ptq$ does not divide $c$, as desired.
\end{proof}


\step Assume $\psi \in \scr B_2$.  Let $\phi \in \irr {K \mid \psi}$ with $\phi(1) = pt$, and let $\chi \in \irr {G \mid \phi}$ with $\chi(1) > \phi(1)$.  Then $\chi(1) = ptf$.

\begin{proof}
By Step 10, $pt \in \cd {K \mid \psi}$, and so, $\phi$ exists.  Also, since $G/K$ is nonabelian, $\chi$ exists by Lemma (3.1) of \cite{partone}.  Now, $\chi(1)/pt$ is divisible by some prime in $\pi(f) \cup \{ q\}$.  By Step 12, we $q$ does not divide $\chi(1)/pt$.  Hence, for some prime $r \in \pi(f)$, the degree $\chi(1)$ is divisible by $ptr$.

As $G/K$ is a Frobenius group, it follows from Theorem 12.4 of \cite{text} that we may choose $c \in \cd {G \mid \phi}$ with $c$ divisible by $ptq$ or $ptf$.  By Step 12, we cannot have $c$ divisible by $ptq$, and so, $ptf$ divides $c$.  This implies that $c$ divides $pt|G:K|$ and $c/pt$ is not
divisible by $q$.  We deduce that $c = ptf$. Now, $\omega(c,\chi(1)) \geq \omega(ptr) = 3$. By the two-prime hypothesis in $G$, we obtain $\chi(1) = c = ptf$, as desired.
\end{proof}


\step $|\cd {G \mid \scr B_2}| \leq 9$.

\begin{proof}
Let $\psi \in \scr B_2$. By Step 10, all the degrees $a \in \cd {K \mid \psi}$ have the form $a = pt$ or $a = pgh_0$ where $h_0 \in \cd {H/L}$.  Recall that the only degree of $G$ lying over $pgh_0$ is $b$, since $b$ is the unique degree of $G$ divisible by $g$. Choose $d \in \cd G$ lying over $pt \in \cd K$ with $d > pt$. By Step 13, $d = ptf$. Thus, $f$ is not prime by Step 11.

If we have $\psi' \in \scr B_2$ with $p' = \psi'(1)$ and $t' = |K_0:I_{K_0}(\psi')|$, then we may construct $d' = p't'f$ as above.  Assume $p=p'$.  By the two-prime hypothesis in $G$, we have $d = d'$.  It follows that for each prime $p \in \pi$ there are at most two irreducible character degrees lying over all $\psi \in \scr B_2$ with $\psi(1) = p$, namely $b$ and $ptf$.  We conclude that $|\cd {G \mid \scr B_2}| \leq |\pi| + 1 \leq 9$.
\end{proof}

Recall that $\scr Y = \scr A \cup \scr B_1 \cup \scr B_2 \cup \scr C$.  By Steps 5, 6, 9, and 14, we compute
$|\cd {G \mid \scr Y}| \leq 166 + 1 + 5,\!521 + 9 = 5,\!697$.
\end{proof}


We are now ready to consider the general case of the $K$ Nonabelian Hypothesis.

\begin{theorem} \label{10.6}
Assume the $K$ Nonabelian Hypothesis. Suppose either $|C:L| \geq r^3$ or there is some prime $s \ne r$ with
$K_0/\Oh sC$ nonabelian.  Then $|\cd G| \leq 462,\!515$.
\end{theorem}

\begin{proof}
In the latter case, let $U \nor G$ with $\Oh sC \sbs U \sbs K$ and let $U$ be maximal with $K/U$ nonabelian.  By Lemma \ref{7.4} (b), we have $(K/U)' = C/U$.  In the first case, simply let $U = L$.  Write $L_0 = U \cap L$.

Note that Lemma \ref{10.1} implies that $K_0/U$ is a Frobenius group with kernel $C/U$.  Thus, $K_0/L_0$ is a Frobenius group with kernel $C/L_0$.  Also, Theorem \ref{10.3} applies to $U$ in place of $L$.

Now, $K_0/L_0$ is a Frobenius group, so if $\psi \in \irr {L_0}$, we have two possibilities by Lemma (2.14) of \cite{partone}: (1) there exists a character $\chi \in \irr {C \mid \psi}$ with $\V \chi \sbs L_0$, or (2) there exists $\chi \in \irr {C \mid \psi}$ with $g \chi(1) \in \cd {K_0}$.  It is possible that $\psi \in \irr {L_0}$ is both of type (1) and of type (2).

If $\psi \in \irr {L_0}$ is of type (2), then $g \psi(1)$
divides some character degree of $G$.  Hence, $g \psi(1)$ divides $b$, the unique irreducible character degree of $G$ divisible by $g$. Since $\subdash b\pi = g$, it must be that $\psi(1)$ is a $\pi$-number.

Let $Y = \{ \psi(1) \mid \psi \in \irr {L_0} \mbox{ is of type (2)}\}$ and let $X = \cd {L_0} \setminus Y$.  Now, $Y$ is a set of $\pi$-numbers, and $|\pi| \leq 8$.  Thus, $\scr G (X)$ has at most 11 connected components by Theorem (2.2) of \cite{partone}.  As $1_{L_0}$ is of type (2), we have $1 \in Y$.

Let $x \in X$, and let $\theta \in \irr {L_0}$ with $\theta(1) = x$.  By definition of $X$, we may choose $\chi \in \irr {C \mid \theta}$ with $\V \chi \sbs L_0$.  Now, $|C:L_0|$ divides
$(\chi(1)/\theta(1))^2$.  If $r^3$ divides $|C:L_0|$, then $r^2$ divides $\chi(1)/\theta(1)$. Otherwise, the distinct primes $r$ and $s$ divide $|C:L_0|$, and we have $rs$ divides $\chi(1)/\theta(1)$.  Write $u = r^2$ in the first case, and $u = rs$ in the second case.

For all characters $\theta \in \irr {L_0}$ with $\theta(1) \in X$ and $\chi \in \irr {C \mid \theta}$, we have that $u$ divides $\chi(1)/\theta(1)$.  Let $x,y \in X$ with $d = (x,y) > 1$.  Then all elements of $\cdover Cx{L_0}$ and $\cdover Cy{L_0}$ are divisible by $du$, and $\omega(du) \geq 3$.  Thus, $|\cdover G{ \{ x,y \} }{L_0}| = 1$, and it follows by induction that if $X_0$ is a component of $\scr G(X)$, then $|\cdover G{X_0}{L_0}| = 1$. Since $\scr G(X)$ has at most 11
connected components, $|\cdover GX{L_0}| \leq 11$.

We now consider the size of the set $\cdover GY{L_0}$.  For each positive integer $i$, write $Y^{(=i)} = \{ y \in Y \mid \omega(y) = i \}$, and $Y^{(\geq i)} = \bigcup_{j \geq i} Y^{(=j)}$.

Let $v \in \pi^3$.  By the two-prime hypothesis in $G$, there
is at most one irreducible character degree of $G$ which lies over all degrees in $\cd {L_0}$ which are divisible by $v$. Since $Y$ is a set of $\pi$-numbers, we can use Lemma \ref{2.7} to see that $| \cdover G {Y^{(\geq 3)}} {L_0} | \leq | \pi^3 | \leq 120$.

Fix $y \in Y^{(=2)}$, and let $z \in \cdover {K_0}y{L_0}$. In
the case where $L = L_0$, we have not defined the prime $s$.  In this case, simply let $s = r$.  Since $\pi(K_0:L_0) = \{ r,s\} \cup \pi(g)$, we have that one of the following occurs:
(1) for some prime $p \in \pi(g)$, we have $py$ divides $z$, (2) $ry$ divides $z$, (3) $sy$ divides $z$, or (4) $z = y$.
For $1 \leq i \leq 4$, write $Z_i$ to be the set of $z \in \cdover {K_0}y{L_0}$ with $z$ satisfying condition $(i)$.  Note that in case (1) we have many possible primes $p$, whereas in cases (2) and (3) we refer to the fixed primes $r$ and $s$ (noting that possibly $s = r$).

Since $y \in Y$, we have $ygr^\alpha s^\beta \in \cd {K_0}$ for some $\alpha, \beta \geq 0$.  Hence, if $z \in Z_1$, then
\[
\omega(z, ygr^\alpha s^\beta) \geq \omega(py) = 3.
\]
since $\omega(y) = 2$.  Let $d \in \cdover Gz{K_0}$, and observe that $z$ divides $d$.  Also, as $g$ divides $ygr^\alpha s^\beta$, it must be that $ygr^\alpha s^\beta$ divides $b \in \cd G$.  By the two-prime hypothesis in $G$, we obtain $d=b$.  We conclude that $\cdover G{Z_1}{K_0} = \{ b \}$.

Note that $\omega (ry) = 3$ and that $ry$ divides all irreducible character degrees of $G$ which lie over degrees in $Z_2$.  By the two-prime hypothesis in $G$, we have $|\cdover G{Z_2}{K_0}| = 1$. Similarly, $|\cdover G{Z_3}{K_0}| = 1$.

If $z \in Z_4$, then $z=y$.  Since $|G:K_0|$ is a $\pi$-number, the degrees in $\cdover Gz{K_0}$ have the form $vy$, where $v$ is a $\pi$-number.  Recall that $y \in Y^{(=2)}$, so $y \in \pi^2$.  This yields $\cdover G{Z_4}{K_0} \sbs \pi^{\geq 2}$.  Using the two-prime hypothesis in $G$ and Lemmas \ref{2.7} and \ref{2.8}, we obtain
\[
|\cdover G{Z_4}{K_0}| \leq |\pi^2|+|\pi^3| \leq 36 + 120 = 156.
\]
Therefore,
\[
|\cdover Gy{L_0}| \leq \sum_{i=1}^4 |\cdover G{Z_i}{K_0}| \leq
 1+1+1+156 = 159.
\]
Since $Y$ is a set of $\pi$-numbers, we can use Lemma \ref{2.7} to see that $|Y^{(=2)}| \leq |\pi^2| \le 36$.  We then obtain
$$
|\cdover G{Y^{(=2)}}{L_0}| \leq \sum_{y \in Y^{(=2)}} |\cdover
Gy{L_0}| \leq |Y^{(=2)}| \cdot 159 \leq 36\cdot 159 = 5,\!724.
$$

Write $\scr Y = \{ \psi \in \irr {L_0} \st \psi(1) \in Y^{(=1)} \}$.  Write $\scr Y_0$ for the set of characters $\psi \in \scr Y$ that extend to $C$.  By Lemma \ref{10.5}, $|\cd {G \mid \scr Y_0}| \leq 5,\!697$. Let $\scr Y_1$ be the set of $\psi \in \scr Y$ that do not extend to $C$.  For each $p \in \pi$, write $\scr Y_1(p)$ for the set of $\psi \in \scr Y_1$ with $\psi(1) = p$.

Let $\psi \in \scr Y_1(p)$. Since $\psi$ does not extend to $C$, $rp$ or $sp$ divides every element of $\cd {C \mid \psi}$.  In fact, $rp$ or $sp$ divides all degrees in $\cd {C \mid \scr Y_1(p)}$. Let $\chi \in \irr {C \mid \scr Y_1(p)}$.  In light of Lemma (2.14) of \cite{partone}, either $g\chi(1) \in \cd {K_0}$ or $up$ divides $\chi(1)$. Recall that, depending on the structure of $C/L_0$, either $u = r^2$ or $u = rs$.  In particular, either $g\chi(1) \in \cd {K_0}$ or $\omega(\chi(1)) \geq 3$.

First, assume that $g \chi(1) \in \cd {K_0}$ and that $\omega (\chi(1)) \leq 2$.  Since $\chi(1)/p$ divides $|C:L_0|$, we have either $\chi(1) = rp$ or $\chi(1) = sp$.  For notational convenience, write $\chi(1) = vp$ where $v \in \{ r, s\}$. Then $g v p = g \chi(1) \in \cd {K_0}$.  Every element of $\cdover {K_0}{\chi(1)}C$ has the form $z\chi(1) = zvp$ where $z$ is some divisor of $g$.  We obtain
\[
\omega(gvp, zvp) \geq 2+\omega(z).
\]
By Lemma (3.6) of \cite{partone}, $K_0$ satisfies the $(2,\pi)$-hypothesis.  Since $g$ is $\pi'$-number, it follows that if $z > 1$, then $z = \subdash z\pi = \subdash g\pi = g$. The only possibilities for $z$ are $z=1$ and $z=g$, and thus, $|\cdover {K_0}{\chi(1)}C| \leq 2$. Recall that $\chi(1) = rp$ or $\chi(1) = sp$.  This makes at most
four degrees of $G$ lying over $\scr Y_1(p)$.

Second, assume $\omega(\chi(1)) \geq 3$. Since $\chi(1)/p$ divides $|C:L_0|$, we have $\chi(1)$ is divisible by $pr^2$, $prs$, or $ps^2$.  By the two-prime hypothesis in $G$, each of these possibilities yields at most one character degree of $G$.  Thus, we have at most three degrees of $G$ lying over such characters $\chi$.

In particular, $|\cd {K_0 \mid \scr Y_1(p)}| \leq 4+3=7$.  Since $|\pi| \leq 8$, we determine that
\[
|\cd {K_0 \mid \scr Y_1}| \leq \sum_{p \in \pi} |\cd {K_0 \mid \scr Y_1(p)}| \leq 8 \cdot 7 = 56.
\]
Similarly, as $|G:K_0|$ is a $\pi$-number,
\[
|\subdash {\cd {G \mid \scr Y_1}}\pi| = |\subdash {\cd {K_0 \mid \scr Y_1}}\pi| \leq |\cd {K_0 \mid \scr Y_1}| \leq 56.
\]
By Lemmas \ref{2.7} and \ref{2.8}, we compute
\[
|\cd {G \mid \scr Y_1}| \leq 56 \cdot |\pi^{\leq 2}| + |\pi^3| \leq 56 \cdot 45 + 120 = 2,\!640.
\]
Hence, $|\cdover G{Y^{(=1)}}{L_0}| \leq |\cd {G \mid \scr Y_0}| + |\cd {G \mid \scr Y_1}| \leq 5,\!703 + 2,\!640 = 8,\!343$.

We consider the degrees of $G$ which lie over $Y^{(=0)} = \{ 1 \}$, namely the set $\cd {G/L_0'}$.  We may assume $L_0$ is abelian and work to obtain a bound on $|\cd G|$.  If $L = L_0$, then Theorem \ref{10.3} gives $|\cd G| \leq 221,\!205$.  We now may assume $L_0 < L$. Note that $\pi(C:L_0) = \{ r, s\}$.

Let $\scr A \sbs \irr C$ be the set of characters with degree
divisible by $rs$, and let $\scr B = \irr {K \mid \scr A}$.  For each character $\psi \in \scr B$, we find a character $\chi \in \irr {G \mid \psi}$ with $\chi(1)/\psi(1)$ divisible by $q$ or $f$ by Theorem (12.4) of \cite{text}.  Now, $\chi(1)$ is divisible by $rsf$ or $rsq$, and so, by the two-prime hypothesis in $G$, there are at most two degrees we can construct in this way.  Since $|G:K|$ is a $\pi$-number, $\subdash {\chi(1)}\pi = \subdash {\psi(1)}\pi$, and so, $|\subdash {\cd {K \mid \scr A}}\pi| \leq 2$.  Again, since $|G:K|$ is a $\pi$-number, $\subdash {\cd {K \mid \scr A}}\pi = \subdash {\cd {G \mid \scr A}}\pi$. Hence, $|\subdash {\cd {G \mid \scr A}}\pi| \leq 2$.  Because $|\pi| \leq 8$, we may use Lemmas \ref{2.7} and \ref{2.8} to compute
\[
|\cd {G \mid \scr A}| \leq 2\cdot |\pi^{\leq 2}| + |\pi^{\leq 3}| \leq 2 \cdot 45 + 120 = 210.
\]

Assume $\theta \in \irr C - \scr A$.  Then $rs$ does not divide $\theta(1)$.  Since $L_0$ is abelian and $|C:L_0|$ is an $\{ r,s \}$-number, we may apply It\^o's Theorem to see that $\theta(1)$ is either a power of $r$ or a power of $s$.  If $\theta(1)$ is a power of $r$, then $\theta_L$ is a sum of linear characters and $\irr {G|\theta} \sbs \irr {G/L'}$.  If $\theta(1)$ is a power of $s$, then $\theta_U$ is a sum of linear characters and $\irr {G \mid \theta} \subseteq \irr {G/U'}$. By Theorem \ref{10.3}, $|\cd {G/L'}| \leq 221,\!205$ and $|\cd {G/U'}| \leq 221,\!205$.

Therefore, we have proved $|\cd {G/L_0'}| \leq 210 + 2\cdot 221,\!205 = 442,\!620$. Finally, we compute:
$$
| \cd G | \leq |\cd{G/L_0'}| + | \cdover GX{L_0} | +
  \sum_{i=1}^\infty | \cdover G{Y^{(=i)}}{L_0} |.$$
We conclude that $|\cd G| \leq
  442,\!620 + 11 + 120+8,\!343 + 5,\!724 + 5,\!697 =  462,\!515.$
\end{proof}

\section{The Small Frobenius Case, part 3}

{\bf Hypothesis (``Small Kernel'')}. Assume the $f$ Small Hypothesis from Section \ref{Chapter 9}. Assume $K$ is nonabelian, and write $\pi = \pi(G:K)$.  Assume $K' \sbs \Oh p K$ for all primes $p$.  Let $L \nor G$ with $L \sbs K$ and suppose $L$ is chosen maximal with $K/L$ nonabelian.  Assume that for all such choices of $L$ we have $|\pi(K:L) - \pi| \geq 9$.  Write $C/L = (K/L)'$ and $K/C = K_0/C \times H/C$, where $K_0/C = \oh {\pi'} {K/C}$ and $H/C = \oh {\pi} {K/C}$.  Let $g = |K_0:C|$ and $h = |H:C|$.  Suppose $L$ is chosen so that $gh$ is minimized.  Write $|C:L| = r^e$ and assume $e \leq 2$. Assume that for all primes $s \not= r$, we have $K_0/\Oh s C$
abelian.  Write $B/L = \cent {G/L} {C/L}$.
\medskip

Note that in Section \ref{Chapter 9}, we assumed that it was
possible to choose an $L$ with $|\pi(K:L) - \pi| \leq 8$.  Here we consider the complementary case.  In the next three results, we work to prove that $C$ is nilpotent, and then apply the results of Section 6 of \cite{partone}.

Under these hypotheses, we have that $K_0/L$ is a Frobenius group with kernel $C/L$ by Lemma \ref{10.1}.


\begin{lemma} \label{11.1}
Assume the Small Kernel Hypothesis.  Then either $KB = N$ or $B = C$. If $KB = N$, then $B/L = C/L \times D/L$ and $N/L = D/L \times K/L$ for some $D \nor G$.
\end{lemma}

\begin{proof}
Clearly either $KB > K$ or $B \sbs K$.  Under the Small Kernel
Hypothesis, we have that $C/L = (K/L)^\infty$ is the unique minimal characteristic subgroup of $K/L$ by Lemma \ref{7.3}. Hence, by Lemma \ref {7.2} (b), $K/C$ has a regular orbit on $C/L$. In particular, $K/C$ acts faithfully on $C/L$, and so $B \cap K = C$. Thus, if $B \sbs K$, then $B = C$ as desired.

Assume $KB > K$.  Since $KB \nor G$ and $N/K$ is the unique minimal normal subgroup of $G/K$, we have $N \sbs KB$.

Observe that $C/L \sbs \zent {B/L}$, and so $K_0/C$ acts on $B/L$.  Note $|B:C| = |KB:K|$ is a $\pi$-number.  Also, $|C:L|$ is coprime to $|K_0:C|$ and so the action is coprime.  Since $G$ is solvable, $B/L = \cent {B/L} {K_0/C} \cdot [B/L, K_0/C]$. We have that $K_0/C$ centralizes $B/C$ since $B \cap K_0 = C$, so $[B/L, K_0/C] \sbs C/L$.  Now, $K_0/C$ acts Frobeniusly on $C/L$, so $[B/L, K_0/C] = C/L$. Write $D/L = \cent {B/L} {K_0/C}$.  Because $K_0/C$ acts Frobeniusly on $C/L$, it follows that $(C/L) \cap (D/L) = 1$.  Thus, $B/L = C/L \times D/L$.

Note that $KB/D \cong K/L$, and so, $gh \in \cd {KB}$.  There is a unique character degree $b \in \cd G$ divisible by $g$, by the two-prime hypothesis.  Observe that $b$ must lie above $gh \in \cd {KB}$, and so, $b$ divides $gh |G:KB|$.  Since $gh \in \cd K$, there exists a degree $d \in \cd G$ lying over $gh \in \cd K$ with $qgh$ dividing $d$ or $fgh$ dividing $d$ by Theorem (12.4) of \cite{text}.   As $g$ divides $d$, we obtain $d=b$.  Hence, $qgh$ divides $b$ or $fgh$ divides $b$.

Now, $q$ does not divide $|G:KB|$ since $KB \sps N$.  It must be that $fgh$ divides $b$.  Since also $b$ divides $gh|G:KB|$ and $|G:KB|$ divides $f$, it follows that $|G:KB| = f$.  Thus, $KB = N$, as desired.

Finally, $N = KB = K(CD) = KD$.  We see that
\[
K \cap D \sbs K \cap B \cap D = C \cap D = L,
\]
and so $N/L = K/L \times D/L$.
\end{proof}

In the proof of the next theorem, we will use some modular character theory.  Let $r$ be a prime.  We use $G^0$ to denote the set of $r$-regular elements of $G$, that is, elements whose order is not divisible by $r$.  If $\chi \in \cha G$, write $\chi^0$ for the restriction of $\chi$ to $G^0$ and note that this is a Brauer character.  We write $\ibr G$ to denote the set of irreducible Brauer characters.


\begin{theorem} \label{11.2}
Assume the Small Kernel Hypothesis.  Assume $B = C$.  Then $C$ is nilpotent.
\end{theorem}

\begin{proof}
First note that $G/C$ is isomorphic to a subgroup of $\aut {C/L}$.  Because $G/C$ is nonabelian, $\aut {C/L}$ is nonabelian.  In particular, $C/L$ is not cyclic.  We assumed $|C:L| \leq r^2$, so $|C:L| = r^2$.  Consider $\oh r{G/C}$. Since $G/C$ acts faithfully on the $r$-group $C/L$, we obtain $\oh r{G/C} = 1$.

Now, $G/C$ is isomorphic to a subgroup of ${\rm GL}(2,r)$.  This isomorphism is a faithful 2-dimensional representation $\scr X$ of $G/C$ in characteristic $r$.  Write ${\bf F}_r$ for the field with $r$ elements.  We want to show that $\scr X$ is absolutely irreducible.  Assume that over $\overline{{\bf F}_r}$, the algebraic closure of ${\bf F}_r$, that $\scr X$ is reducible.  Then since $\scr X$ is a two-dimensional representation, over $\overline{{\bf F}_r}$ we have that $\scr X(g)$ is a diagonal matrix for all $g \in {G/C}$.  However, $\scr X$ is faithful and $G/C$ is nonabelian, a contradiction.  Therefore, $\scr X$ is absolutely irreducible.

Let $\phi \in \ibr {G/C}$ be afforded by $\scr X$. Note that $\phi$ is faithful of degree 2.  Since $G/C$ is solvable, we may apply the Fong-Swan Theorem, (10.1) from \cite{navarro}.  In light of the Fong-Swan Theorem, there exists a character $\chi \in \irr G$ with $\chi^0 = \phi$.  Let $x \in \ker \chi$.  If $x$ is an $r$-regular element, then $\phi(x) = \chi(x) = \chi(1) = \phi(1)$.  However, since $\phi$ is faithful, it
follows that the only $r$-regular element in $\ker (\chi)$ is $1$.  Hence, all the nonidentity elements of $\ker (\chi)$ have order divisible by $r$. It follows that $\ker (\chi)$ is an $r$-group, and so ,$\ker (\chi) \sbs \oh r{G/C} = 1$. We conclude that $\chi$ is faithful. We now work to prove that $r$ does not divide $|G/C|$.

First, assume $\chi$ is primitive.  Then $G$ is a linear group, and Theorem (14.23) of \cite{text} implies that $|(G/C): \zent{G/C}| \in \{ 12, 24, 60 \}$.  Write $Z/C = \zent {G/C}$.  Now, $g$ divides $r^2 - 1$, and $|\pi (g)| \geq 8$.
Clearly, $r$ is large enough so that $r$ does not divide $|G:Z|$.  Hence, $|G/C|_r$ divides $|Z/C|$.  In particular, $G/C$ has a normal Sylow $r$-subgroup and since $\oh r{G/C} = 1$, we determine that $r$ does not divide $|G/C|$.

Second, assume $\chi$ is imprimitive.  Then there is a subgroup $M < G$ and character $\theta \in \irr M$, so that $\theta^G = \chi$.  Since $\chi(1) = 2$, it must be that $\theta$ is linear and $|G:M| = 2$.  Hence, $M \nor G$, and also, $M' \sbs \ker (\chi) \sbs C$. We deduce that $M/C$ is abelian.  As before, $r$ is a large prime, so certainly $r > 2$, and $|G/C|_r$ divides $|M/C|$. In particular, $G/C$ has a normal Sylow
$r$-subgroup, and we conclude that $r$ does not divide $|G/C|$.

We want to apply Theorem \ref{7.7} with the notation $\Gamma =
G$, $G = K_0$, $N = C$, and $K = L$.  Clearly, $G$ satisfies the two-prime hypothesis. Hypotheses (2) - (5) are satisfied
since they appear in the Small Kernel Hypothesis.  We have just proved that $r$ does not divide $|G:C|$, so (7) follows.

For (6), we want to prove for every $S \nor G$ with $C/S$ an abelian $r$-group and $\cent {C/S}{K_0/S} = N/S$ that the
action of $K_0/C$ on $C/S$ is Frobenius.  Let us check the
hypotheses of Lemma \ref{7.4} (c) applied to the group $K_0$. Assume for some prime $p$ that we have $K_0' \not\sbs \Oh p{K_0}$. Then $K/\Oh p{K_0}$ is nonabelian, and we may choose $M \nor G$ with $\Oh p{K_0} \sbs M \sbs K$ such that $M$ is maximal with $K/M$ nonabelian.  In this case, $|\pi(K:M) \setminus \pi| \leq 1$ which contradicts our hypothesis.

We claim that $L$ is maximal among normal subgroups of $G$ with $K_0/L$ nonabelian.  Suppose $L < M \sbs K_0$ where $M \nor G$. By the choice of $L$, we have that $K/M$ is abelian, and hence
$K_0/M$ is abelian, as desired.

We want to show that $L$ is chosen in $K_0$ such that $g$ is
minimized.  Suppose $M \nor G$ with $M \sbs K_0$ and $M$ chosen maximal with $K_0/M$ nonabelian, and assume $|(K_0/M):(K_0/M)'| < g$.  Let $\tilde M \nor G$ with $M \sbs \tilde M \sbs K$ and suppose $\tilde M$ is maximal with $K/\tilde M$ nonabelian.  Write $E/\tilde M = (K/\tilde M)'$. Now $K_0E /E \cong K_0 / (K_0 \cap E)$ is abelian, and so $(K_0/M)' \sbs E/M$.  We have
\[
|K:E| \leq |K:K_0| \cdot |K_0:K_0 \cap E| \leq |K:K_0|\cdot
|(K_0/M):(K_0/M)'| < gh.
\]
This contradicts the assumption that $L$ was chosen with $gh$
minimized.

Finally, note that by Lemma \ref{10.1}, $K_0/L$ is a
Frobenius group, and so $K_0/C$ is cyclic.  Hence, Lemma \ref{7.4} (c) applies to $K_0/L$, which gives condition (6) of Theorem \ref{7.7}.  Thus, Theorem \ref{7.7} applies and $C$ is nilpotent, as desired.
\end{proof}


We now consider the remaining case and show that $C$ is nilpotent.  It happens that it is easier to prove that $B$ is nilpotent, and since $C \sbs B$ the result will follow.

\begin{theorem} \label{11.3}
Assume the Small Kernel Hypothesis.  Then $C$ is nilpotent.
\end{theorem}

\begin{proof}
By Lemma \ref{11.1} and Theorem \ref{11.2}, we may assume $KB = N$.  Let $D$ be as in the statement of Lemma \ref{11.1}. We want to apply Theorem \ref{7.7} with the notation $\Gamma = G$, $G = K_0 B$, $N = B$, and $K = D$.

By Lemma \ref{11.1}, we have $N/L = D/L \times K/L$.  Now, $K_0 B = D/L \times K_0/L$, and so $K_0 B / D \cong K_0/L$ is a Frobenius group by Lemma \ref{10.1}.  Hence, conditions (1)-(4) of Theorem \ref{7.7} are satisfied.

Write $b$ for the unique character degree of $G$ divisible by $g$, where uniqueness follows by the two-prime hypothesis.  Now, $gh \in \cd {K/L}$, and since $K/L \cong N/D$, we have $gh \in \cd N$. Hence, $b$ divides $gh|G:N| = fgh$.  Since $gh \in \cd K$, there exists $d \in \cdover G{gh}K$ with $qgh$ dividing $d$ or $fgh$ dividing $d$ by Theorem 12.4 of \cite{text}.  Since $g$ divides $d$, it follows that $d = b$ and $qgh$ divides $b$ or $fgh$ divides $b$.  We must have $b = fgh$.

Let $1 \not= \lambda \in \irr {B/D}$.  Write $T = I_G(\lambda)$.  Now, $T \cap K_0 B = B$ since $K_0 B/D$ is a Frobenius group.  Hence, $N \cap T \nor T$ is an $r'$-group and $T/(T\cap N)$ is cyclic.  Thus, $\lambda$ extends to $T$.  We obtain $|G:T| \in \cd G$, and $g$ divides $|G:T|$.  We deduce that $|G:T| = b = fgh = |G:B|$.  We see that $T = B$ and
$G/B$ acts Frobeniusly on $B/D$.  In particular, $r$ does not divide $|G:B|$, which gives (7).

We now work to prove condition (5), namely that $K_0 B / \Oh p B$ is abelian for all primes $p \not= r$.  First, let $p \notin \{ q,r \}$ be prime.  Write $U = \Oh p B$ and $V = U \cap C$.  Since $|B:U|$ and $|B:C|$ are coprime, $UC = B$.  Hence, $K_0 B/U \cong K_0/V$.  We obtain $\Oh p C \sbs V$, and since
$p \not= r$, we have $K_0/V$ is abelian by assumption.  Thus $K_0 B/U$ is abelian, as desired.

Assume $q \not= r$, and write $U = \Oh q B$.  Since $B/C \cong
N/K$ is a $q$-group, $U \sbs C$.  In particular, $\Oh q C
\sbs U$.  We next want to show that $B/U$ is abelian.  Let $x \in \cd {B/U}$, and let $\psi \in \irr {B/U}$ with $\psi(1) = x$.  Now, $G/D$ is a Frobenius group with kernel $B/D$, and by Theorem 12.4 of \cite{text}, we have either $\V \psi \sbs D$ or $|G:B|x \in \cd G$.  As $x$ is a power of $q$ and $|B:D|$ is a power of $r$, we cannot have $\V \psi \sbs D$, and $|G:B|x \in \cd G$.  Now, $|G:B| = fgh$, so $g$ divides $|G:B|x$.  It follows that $|G:B|x = b = fgh$, which implies $x=1$, as desired.

By assumption, since $\Oh q C \sbs U$ and $q\not= r$, we have
$K_0/U$ abelian.  Now, $C/U$ is a $q$-group and $q$ does not divide $g$.  Write $K_0/U = C/U \times E/U$, where $|E/U| = g$.  It follows that $K_0 B/U = B/U \times E/U$, and so $K_0 B/U$ is abelian, as desired. This proves (5).

For (6), assume $S \nor G$ such that $B/S$ is an abelian $r$-group, and suppose that $\cent {B/S} {K_0 B / B} = 1$.  We want to show that the action of $K_0 B/B$ on $B/S$ is Frobenius.  First, assume $S \sbs C$.  Then $[B/S, K_0 B / B] \sbs C/S$ since $K_0 B / C$ is abelian. But now $B/S = \cent {B/S} {K_0 B / B} \cdot [B/S, K_0 B / B] \sbs C/S$, which is a contradiction.

Thus, we may assume $S \not\sbs C$.  Now, $SC > C$, and since $B/C \cong N/K$ is a $G$-chief factor, we must have $SC = B$.  It follows that $K_0 B / S \cong K_0 / S \cap C$.  We see that $C/(S\cap C)$ is an abelian $r$-group, and the action of $K_0/C$ on $C/(S\cap C)$ is permutation isomorphic to the action of $K_0 B/B$ on $B/S$. As in the proof of Theorem \ref{11.2}, Lemma \ref{7.4} (c) applies to the group $K_0/L$.  Hence, $K_0/C$ acts Frobeniusly on $C/(S \cap C)$, and so $K_0 B/B$ acts Frobeniusly on $B/S$.

Therefore, Theorem \ref {7.7} applies, and implies that $B$ is nilpotent.  Since $C \sbs B$, we have that $C$ is nilpotent.
\end{proof}

We now obtain our bound on $|\cd G|$ under the Small Kernel Hypothesis.

\begin{theorem} \label{11.4}
Assume the Small Kernel Hypothesis. Then $|\cd G| \leq 221,\!205$.
\end{theorem}

\begin{proof}
By Theorem \ref{11.3}, $C$ is nilpotent.  It follows that
$|G:K_0|$ is a $\pi$-number, and so by Lemma (3.6) of \cite{partone}, $K_0$ satisfies the $(2,\pi)$-hypothesis.  Since $K_0/L$ is a Frobenius group with kernel $C/L$, we apply Theorem (6.3) of \cite{partone} to see that $|\subdash {\cd {K_0}}\pi| \leq (1 + 2 ^{2 \cdot 2})^3 = 17^3 = 4,\!913$.  Since $|G:K_0|$ is a $\pi$-number, $\subdash {\cd G}\pi = \subdash {\cd {K_0}}\pi$.  Hence, Lemmas \ref{2.7} and \ref{2.8} yield
\[
|\cd G| \leq 4,\!913 \cdot |\pi^{\leq 2}| + |\pi^3| \leq 4,\!913 \cdot 45 + 120 = 221,\!085 + 120 = 221,\!205.
\]
\end{proof}

Finally, we are ready to prove Theorem A.



\begin{proof}[Proof of Theorem A]
If $G$ is abelian, then $|\cd G| = 1$.  Assume $G$ is nonabelian and let $K \nor G$ be maximal with $G/K$ nonabelian. By the Two-Case Theorem, either $G/K$ is a $p$-group, or $G/K$ is a Frobenius group with kernel $N/K$.  If $G/K$ is a $p$-group, then by Theorem~B, $|\cd G| \leq 88$.  Hence, we may assume $G/K$ is a Frobenius group and $\Oh p G \sbs G'$ for all primes $p$.  We may also assume that we have chosen $K$ such that $f=|G:N|$ is minimized.

If $|\pi(f)| \geq 8$, then we have the $f$ Large Hypothesis and Theorem \ref{8.3} implies that $|\cd G| \leq 4,\!913$.  Thus, we assume $|\pi (f)| \leq 7$, and we are in the case of the $f$ Small Hypothesis. If $K$ is abelian, then by Theorem \ref{9.2} we have $|\cd G| \leq 208$.  Hence, we may assume $K$ is nonabelian.  Let $L \nor G$ with $L \sbs K$ and suppose $L$ is maximal with $K/L$ nonabelian. Assume further that $L$ is chosen so that $|K/L : (K/L)'|$ is minimized.

If it is possible to choose $L$ such that $|\pi(K:L) \setminus \pi| \leq 8$, then Theorem \ref{9.3} implies that $|\cd G| \leq 27,\!540$.  Hence, we assume that $|\pi(K:L) \setminus \pi| \geq 9$ and that this inequality holds for all choices of $L$.

Write $C/L = (K/L)'$ and $N/C = K_0/C \times H/C$ where $K_0/C$ is a $\pi'$-group and $H/C$ is a $\pi$-group.  If either
$\mbox{rank}(C/L) \geq 3$ or there is some prime $s$ not dividing $|C:L|$ with $K_0/\Oh s C$ nonabelian, then we have the $K$ Nonabelian Hypothesis. By Theorem \ref{10.6}, we obtain $|\cd G| \leq 462,\!515$. Otherwise, we have the Small Kernel Hypothesis, and by Theorem \ref{11.4}, we conclude that $|\cd G| \leq 221,\!205$.

This exhausts all cases and completes the proof.
\end{proof}

\end{document}